\documentclass[a4paper,10pt,leqno]{amsart}
\usepackage{amsmath,amsfonts,amssymb,amsthm,epsfig,epstopdf,url,array}
\usepackage{amsthm}
\usepackage{mathrsfs}
\usepackage{epsfig}
\usepackage{graphicx}

\newtheorem{thm}{Theorem}[section]
\newtheorem{prop}[thm]{Proposition}

\newtheorem{lem}[thm]{Lemma}
\newtheorem{rmk}[thm]{Remark}
\newtheorem{defn}{Definition}[section]

\numberwithin{equation}{section}
\newcommand{\sgn}{\text{sgn}}
\newcommand{\ord}{\text{ord}}
\newcommand{\N}{{\mathbb{N}}}
\newcommand{\z}{{\mathbb{Z}}}

\newcommand{\floor}[1]{\left\lfloor #1 \right\rfloor}
\newcommand{\ceil}[1]{\left\lceil #1 \right\rceil}
\newcommand{\eq}{\equiv}

\title[Regular $m$-gonal forms]{Regular $m$-gonal forms}

\author{BYEONG MOON KIM and Dayoon Park}
\address{Department of Mathematics, Kangnung National University, Kangnung, 210-702, Korea}
\email{kbm@kangnung.ac.kr}
\thanks{}

\address{Department of Mathematics, The University of Hong Kong, Hong Kong}
\email{pdy1016@hku.hk}
\thanks{}

\begin{document}
\maketitle

\begin{abstract}
In this paper, we show that for a fixed rank $n$, there are only finitely many $m$ for which there is a regular $m$-gonal form of rank $n$ and determine every type of the (generalized) regular $m$-gonal form for every sufficiently large $m$.
\end{abstract}

\maketitle
\section{Introduction}
\vskip 0.3cm
\vskip 0.3cm
The {\it $x$-th $m$-gonal number} is defined as the total number of dots to constitute a regular $m$-gon with $x$ dots for each side, which follows the formula 
\begin{equation}\label{m-number}P_m(x)=\frac{m-2}{2}x^2-\frac{m-4}{2}x \end{equation}
for $x \in \N$ and $m \ge 3$.
We call a weighted sum of $m$-gonal numbers 
\begin{equation}\Delta_{m,\mathbf a}(\mathbf x):=a_1P_m(x_1)+\cdots+a_nP_m(x_n) \end{equation}
with $\mathbf a \in \N^n$ as {\it $m$-gonal form}.
Representability of positive integer by an $m$-gonal form has been studied for a long time.
Fermat famously conjectured that every positive integer is written as a sum of at most $m$ $m$-gonal numbers, on the other words,
$$\Delta_{m,(1,\cdots,1)}(x_1,\cdots,x_m)=N$$ has a non-negative integer solution $(x_1,\cdots,x_m) \in \N_0^m$ for any $N \in \N$.
The celebrated Lagrange's four square Theorem and Guass's Eureka Theorem resolve his conjecture for $m=4$ and $3$, respectively.
And Cauchy completely proved his conjecture in 1813. 

On the other hand, by admitting the variable $x$ in the formula (\ref{m-number}) zero and negative integers too, we may generalize the $m$-gonal numbers. 
We call a weighted sum of generalized $m$-gonal numbers (i.e., which admit every integer variable $x_i \in \z$)
$$\Delta_{m,\mathbf a}(\mathbf x):=a_1P_m(x_1)+\cdots+a_nP_m(x_n)$$
with $\mathbf a \in \N^n$ as {\it generalized $m$-gonal form}.
We call an $m$-gonal form (resp, generalized $m$-gonal form) $\Delta_{m,\mathbf a}(\mathbf x)$ {\it represents} $N \in \N_0$ when the diophantine equation $$\Delta_{m,\mathbf a}(\mathbf x)  =  N $$ has a non-negative integer solution $\mathbf x \in \N_0^n$ (resp, integer solution $\mathbf x \in \z^n$)
and a (generalized) $m$-gonal form $\Delta_{m,\mathbf a}(\mathbf x)$ {\it locally represents} $N \in \N_0$ when $$\Delta_{m,\mathbf a}(\mathbf x)  \eq  N \pmod{r}$$ has an integer solution $\mathbf x \in \z^n$ for every $r \in \z$, or equivalently, $\Delta_{m,\mathbf a}(\mathbf x)=N$ is solvable over the ring of $p$-adic integers $\z_p$ for every prime $p$.
If a (generalized) $m$-gonal form $\Delta_{m,\mathbf a}(\mathbf x)$ represents $N$, then the (generalized) $m$-gonal form $\Delta_{m,\mathbf a}(\mathbf x)$ well locally represents $N$.
On the other hand, there are a bunch of examples which indicate that the converse does not hold.
Especially, we call a (generalized) $m$-gonal form is {\it regular} if the converse also always holds, i.e., if the (generalized) $m$-gonal form represents every non-negative integer $N \in \N_0$
which is locally represented by the form.
The regular form was firstly studied by Dickson in \cite{D} where the term regular was coined.
To avoid superfluous duplication of (generalized) regular $m$-gonal form, we call a (generalized) $m$-gonal form $\Delta_{m,\mathbf a}(\mathbf x)$ satisfying $(a_1,\cdots,a_n)=1$ is {\it primitive} and we consider only (generalized) primitive regular $m$-gonal forms.
Jones and Pall \cite{JP} classified 102 all (generalized) primitve regular ternary square forms ($4$-gonal forms of rank $3$).
Watson \cite{W1} showed that there are only finitely many primitive positive definite regular ternary quadratic forms up to isometry.
And he suggested a transformation on quadratic form called as {\it Watson's $\Lambda$-transformation} which was orginally designed for a quadratic form having smaller class number and nicer local structure than before transforming.
Since then, the Watson's $\Lambda$-transformation is practically being applied in lots of studying regular forms.
By using the Watson's $\Lambda$-transformation, Jagy, Kaplansky and Schiemann presented 913 actual all candidates for primitive positive definite regular ternary quadratic forms up to isometry and showed that all of them but 22 are indeed regular.
Oh confirmed $8$ more on the list \cite{O1}.
And Lemke Oliver \cite{Oliver} finally proved that the remaining $14$ candidates are regular under the assumption of generalized Riemann Hypothesis.

Chan and Oh \cite{CO} showed that there are only finitely many generalized primitive regular ternary triangular forms.
Chan and Ricci \cite{CR} improved the result by showing that there are only finitely many primitive regular ternary quadratic polynomials with a fixed conductor.
Kim and Oh \cite{M-KO} classified every generalized primitive regular ternary triangular form.
He and Kane \cite{HK} proved that there is only finitely many generalized primitive regular ternary polygonal forms.
Which implies that there are only finitely many $m$ for which there is a generalized regular ternary $m$-gonal form.
And then one may naturally wonder how about quaternary, quinary, senary, $\cdots$?, although it is already well known that there are infinitely many generalized primitive regular polygonal forms of rank $n \ge 4$.

In this paper, we show that for each rank $n$, there are only finitely many $m$ for which there is a (generalized) regular $m$-gonal form of rank $n$ and
completely determine the types of all (generalized) regular $m$-gonal forms of rank $n \ge 4$ for $m\ge 14$ with $m \not\eq 0 \pmod{4}$ and $m \ge 28$ with $m \eq 0 \pmod{4}$ by using the arithmetic theory of quadratic forms.
We conventionally adopt a notation $$\left<a_1,\cdots, a_n\right>$$ for a square form (or a diagonal quadratic form) $\Delta_{4,\mathbf a}$.
Any unexplained notation and terminology can be found in \cite{O}. 
\vskip 1em
This paper is organized as follows.
In Section 2, we introduce a new quadratic polynomial called as a generalized shifted $m$-gonal form.
In Section 3, we study a transformation on generalized shifted $m$-gonal form which is mimicked by Watson's $\Lambda$-trnasformation \cite{W1}
and characterize a local structure of generalized regular $m$-gonal forms.
In Section 4, based on the results of Section 3, we classify every (generalized) regular $m$-gonal form of rank $n \ge 4$ for $m\ge 14$ with $m \not\eq 0 \pmod{4}$ and $m \ge 28$ with $m \eq 0 \pmod{4}$.
The main Theorems in this paper are followings.

\begin{thm} \label{main}
\begin{itemize}
\item[(1)] 
For $m\ge 14$ with $m\not\equiv0 \pmod{4}$ and $m-2\equiv0 \pmod{3}$, a generalized primitive regular shifted $m$-gonal form of rank $n \ge 4$ is universal. 

\vskip 0.7em

\item[(2)] For $m\ge 14$ with $m\not\equiv0 \pmod{4}$ and $m-2\not\equiv0 \pmod{3}$, a generalized primitive regular $m$-gonal form $\Delta_{m,\mathbf a}(\mathbf x)$ is 
\begin{equation}\label{0.notuni/3}\text{generalized universal $m$-gonal form, }\end{equation}
or of the form of
\begin{equation}\label{1.notuni/3}\left<a_1,3a_2',\cdots,3a_n'\right>_m\end{equation}
where $a_1 \in \{1,2\}$ which represents every non-negative integer in $a_1+3\N_0 \cup 3\N_0$.

\vskip 0.7em

\item[(3)] For $m \ge 28$ with $m \equiv0 \pmod{4}$ and $m \equiv 2 \pmod{3}$, a generalized primitive regular $m$-gonal form $\Delta_{m,\mathbf a}(\mathbf x)$ of rank $n\ge4 $ is 
\begin{equation}\label{0.notuni/4m8}\text{generalized universal $m$-gonal form,}\end{equation}
or of the form of
\begin{equation}\label{1.notuni/4m8} \left<a_1,4a_2',\cdots,4a_n'\right>_m\end{equation}
where $a_1 \in \{1,3\}$ which represents every non-negative integer in $a_1+4\N_0 \cup 4\N_0$, 
\begin{equation}\label{3.notuni/4m8}\left<1,3,4a_3',\cdots,4a_n'\right>_m\end{equation}
which represents every non-negative integer in $\N_0 \setminus 2+4\N_0$,
or
\begin{equation}\label{4.notuni/4m8} \left<1,1,4a_3',\cdots,4a_n'\right>_m \end{equation}
which represents every non-negative integer in $\N_0 \setminus 3+4\N_0$.

\vskip 0.7em

\item[(4)] For $m\ge28$ with $m\eq 0\pmod{4}$ and $m \not\eq 2\pmod{3}$, a generalized primitive regular $m$-gonal form $\Delta_{m,\mathbf a}(\mathbf x)$ is 
\begin{equation}\label{0.notuni/4m8,1m3}\text{generalized universal $m$-gonal form,}\end{equation}
or of the form of
\begin{equation}\label{1.notuni/4m8,1m3}\left<a_1,3a_2',\cdots,3a_n'\right>_m\end{equation}
where $a_1 \in \{1,2\}$ which represents every non-negative integer in $a_1+3\N_0 \cup 3\N_0$, 
\begin{equation}\label{3.notuni/4m8,1m3} \left<a_1,4a_2',\cdots,4a_n'\right>_m\end{equation}
where $a_1 \in \{1,3\}$ which represents every non-negative integer in $a_1+4\N_0 \cup 4\N_0$, 
\begin{equation}\label{5.notuni/4m8,1m3}\left<1,3,4a_3',\cdots,4a_n'\right>_m\end{equation}
which represents every non-negative integer in $\N_0 \setminus 2+4\N_0$,
\begin{equation}\label{6.notuni/4m8,1m3}\left<1,1,4a_3',\cdots,4a_n'\right>_m\end{equation}
which represents every non-negative integer in $\N_0 \setminus 3+4\N_0$,
\begin{equation}\label{7.notuni/4m8,1m3}\left<3,4,12a_3',\cdots,12a_n'\right>_m\end{equation}
which represents every non-negative integer $n$ with $n \equiv 0,3,4,7 \pmod{12}$,
\begin{equation}\label{8.notuni/4m8,1m3} \left<3,8,12a_3',\cdots,12a_n'\right>_m \end{equation}
which represents every non-negative integer $n$ with $n \equiv 0, 3,8,11  \pmod{12}$, or
\begin{equation}\label{9.notuni/4m8,1m3}\left<3,3,4,12a_4',\cdots,12a_n'\right>_m\end{equation}
which represents every non-negative integer $n$ with $n \equiv 0, 3,4,6,7,10 \pmod{12}$.
\end{itemize}
\end{thm}

\section{generalized shifted $m$-gonal form}
\begin{defn}
We define {\it generalized $x$-th shifted $m$-gonal number of level $r$} as
\begin{equation}P_m^{(r)}(x):=\frac{m-2}{2}x^2-\frac{m-2-2r}{2}x \end{equation}
for $x \in \z$ and integer $r \in [0,m-2]$ satisfying $(r,m-2)=1$.
And we call a weighted sum of generalized shifted $m$-gonal numbers 
\begin{equation}\Delta_{m,\mathbf a}^{(\mathbf r)}(\mathbf x):=a_1P_m^{(r_1)}(x_1)+\cdots+a_nP_m^{(r_n)}(x_n)\end{equation} where $\mathbf a \in \N^n$ as {\it generalized shifted $m$-gonal form}.
\end{defn}

We may see that generalized $m$-gonal number is generalized shifted $m$-gonal number of level $1$ and generalized shifted $m$-gonal numbers have non-negative integer values.
We use notations
$$\left<a_1^{(r_1)},\cdots,a_n^{(r_n)}\right>_m$$ for a generalized shifted $m$-gonal form $\Delta_{m,\mathbf a}^{(\mathbf r)}=\Delta_{m,(a_1,\cdots,a_n)}^{(r_1,\cdots,r_n)}$ and
$$\left<a_1,\cdots,a_n\right>_m:=\left<a_1^{(1)},\cdots,a_n^{(1)}\right>_m$$ for a generalized $m$-gonal form $\Delta_{m,\mathbf a}=\Delta_{m,(a_1,\cdots,a_n)}.$

We call a generalized shifted $m$-gonal form $\Delta_{m,\mathbf a}^{(\mathbf r)}(\mathbf x)$ {\it (globally) represents} $N \in \N_0$ if there is an integer solution $\mathbf x \in \z^n$ for $\Delta_{m,\mathbf a}^{(\mathbf r)}(\mathbf x)=N$ and {\it locally represents} $N \in \N_0$ if there is a $p$-adic integer solution $\mathbf x \in \z_p^n$ for $\Delta_{m,\mathbf a}^{(\mathbf r)}(\mathbf x)=N$ over $\z_p$ for every prime $p$.
And we say that a generalized shifted $m$-gonal form $\Delta_{m,\mathbf a}^{(\mathbf r)}(\mathbf x)$ is {\it regular} if  $\Delta_{m,\mathbf a}^{(\mathbf r)}(\mathbf x)$ (globally) represents every non-negative integer which is locally represented by $\Delta_{m,\mathbf a}^{(\mathbf r)}(\mathbf x)$.
If $\Delta_{m,\mathbf a}^{(\mathbf r)}(\mathbf x)=N$ has a $p$-adic integer solution $\mathbf x \in \z_p^n$, then we call $\Delta_{m,\mathbf a}^{(\mathbf r)}(\mathbf x)$ {\it represents $N \in \z_p$ over $\z_p$} and if $\Delta_{m,\mathbf a}^{(\mathbf r)}(\mathbf x)$ represents every $p$-adic integer $N \in \z_p$ over $\z_p$, then we say that $\Delta_{m,\mathbf a}^{(\mathbf r)}(\mathbf x)$ is {\it universal over $\z_p$}, or simply, {\it $\z_p$-universal}.
If $\Delta_{m,\mathbf a}^{(\mathbf r)}(\mathbf x)$ is $\z_p$-universal for every prime $p$, then we say that $\Delta_{m,\mathbf a}^{(\mathbf r)}(\mathbf x)$ is {\it locally universal} and if $\Delta_{m,\mathbf a}^{(\mathbf r)}(\mathbf x)$ represents every non-negative integer, then we say that $\Delta_{m,\mathbf a}^{(\mathbf r)}(\mathbf x)$ is {\it globally universal}, or simply, {\it universal}.

\begin{rmk}
For each fixed level $r \in [0,m-2]$ (effectively, $r \in (0,m-2)$), except $P_m^{(r)}(0)=0$, there are only two generalized shifted $m$-gonal numbers of level $r$  
$$P_m^{(r)}(1)=r , \  P_m^{(r)}(-1)=m-2-r$$
which are less than $m$ and they satisfy following inequality 
$$\begin{cases}
P_m^{(r)}(1)<\frac{m-2}{2}<P_m^{(r)}(-1) & \text{when } r<\frac{m-2}{2} \\
P_m^{(r)}(-1)<\frac{m-2}{2}<P_m^{(r)}(1) & \text{when } r>\frac{m-2}{2},
\end{cases}
$$
which implies that there is exactly one generalized shifted $m$-gonal numbers of level $r$ which is less than or equal to $\frac{m-2}{2}$ except $P_m^{(r)}(0)=0$.
More generally, for any level $r \in [0,m-2]$, 
$$P_m^{(r)}(-x)+P_m^{(r)}(x)=P_m(-x)+P_m(x)$$ holds for $x \in \N$.
Therefore the number of generalized $m$-gonal numbers of level $r$ which are less than or equal to $\frac{P_m(-x)+P_m(x)}{2}$ except $P_m^{(r)}(0)=0$ would be $2x-1$ for $x \in \N$.

For each fixed $m$, there are $\phi(m-2)$ generalized shifted $m$-gonal numbers having the different levels where $\phi$ is the Euler's phi function.
The generalized shifted $m$-gonal numbers of level $r$ and generalized shifted $m$-gonal numbers of level $m-2-r$ with $(m-2,r)=(m-2,m-2-r)=1$ represent exactly same integers because $P_m^{(r)}(-x)=P_m^{(m-2-r)}(x)$. 
We say that two generalized shifted $m$-gonal numbers $P_m^{(r)}(x)$ and $P_m^{(r')}(x)$ have the {\it same type} if $P_m^{(r)}(x)$ and $P_m^{(r')}(x)$ represents exactly same integers.
We may regard that there are $\frac{\phi(m-2)}{2}$ different types of generalized shifted $m$-gonal numbers for $m \ge 5$.
And we may see that there is only one type of generalized shifted $m$-gonal number which have the same type with the generalized shifted $m$-gonal number of level 1, i.e., generalized $m$-gonal number for $m \in \{3,4,5,6,8\}$ and 
there are at least two different types of the generalized shifted $m$-gonal numbers for $m \not\in \{3,4,5,6,8\}$.
\end{rmk}

\begin{prop} \label{prop1}
Let $\Delta_{m,\mathbf a}^{(\mathbf r)}(\mathbf x)$ be a generalized primitive shifted $m$-gonal form.
\begin{itemize}
\item[(1)] $\Delta_{m,\mathbf a}^{(\mathbf r)}(\mathbf x)$ is $\z_p$-universal for a prime $p$ with $(p,2(m-2)) \not= 1$.
\item[(2)] When $m \not\eq 0\pmod{4}$, $\Delta_{m,\mathbf a}^{(\mathbf r)}(\mathbf x)$ is $\z_2$-universal.
\end{itemize}
\end{prop}
\begin{proof}
(1) Since $(r_i,m-2)=1$, we may use the Hensel's Lemma (see 13:9 in \cite{O}) to show that  for any $N \in \z_p$, $$P_m^{(r_i)}(x_i)=\frac{m-2}{2}x_i^2-\frac{m-2-2r_i}{2}x_i=N$$
has a $p$-adic integer solution $x_i \in \z_p$ over $\z_p$.
From the primitivity of $\Delta_{m,\mathbf a}^{(\mathbf r)}(\mathbf x)$, we may obtain that $\Delta_{m,\mathbf a}^{(\mathbf r)}(\mathbf x)$ is universal over $\z_p.$ 

(2) When $m \not\eq 0\pmod{4}$, for any $N \in \z_2$,
$$P_m^{(r_i)}(x_i)=\frac{m-2}{2}x_i^2-\frac{m-2-2r_i}{2}x_i=N$$
has integer solution $x_i = \frac{(m-2-2r_i) - x_i'}{2(m-2)}\in \z_2$ where $x_i' \in \z_2$ is the $2$-adic integer satisfying $(x_i')^2={(m-2-2r_i)^2-8N(m-2)}$.
From the primitivity of $\Delta_{m,\mathbf a}^{(\mathbf r)}(\mathbf x)$, we may obtain that $\Delta_{m,\mathbf a}^{(\mathbf r)}(\mathbf x)$ is $\z_2$-universal. 
\end{proof}

\begin{rmk} \label{rmk1}
By Proposition \ref{prop1}, a generalized primitive shifted $m$-gonal form $\Delta_{m,\mathbf a}^{(\mathbf r)}(\mathbf x)$ represents every $p$-adic integer over $\z_p$ when $p$ is a prime with $(p,2(m-2))\not=1$(resp, $(p,\frac{m-2}{2})\not=1$) when $m\not\eq 0\pmod{4}$(resp, $m \eq 0\pmod{4}$).
On the other hand, the equality
\begin{equation}\label{m4}\Delta_{m,\mathbf a}^{(\mathbf r)}(\mathbf x)=\frac{m-2}{2}\Delta_{4,\mathbf a}(\mathbf x')-\sum_{i=1}^n a_i \frac{(m-2-2r_i)^2}{8(m-2)}
\end{equation}
where $\mathbf x'=\mathbf x-\left(\frac{m-2-2r_1}{2(m-2)},\cdots,\frac{m-2-2r_n}{2(m-2)}\right)$ may help us to characterize the $p$-adic integers which are represented by $\Delta_{m,\mathbf a}^{(\mathbf r)}(\mathbf x)$ over $\z_p$ for a prime $p$ with $(p,2(m-2))=1$ (resp, $(p,\frac{m-2}{2})=1$) when $m\not\eq 0\pmod{4}$ (resp, $m \eq 0\pmod{4}$) with some well known facts on the theory of quadratic form over local ring for $\frac{m-2-2r_i}{2(m-2)}, \frac{(m-2-2r_i)^2}{8(m-2)} \in \z_p$ for all $i$. 
The equality (\ref{m4}) suggests that the $\z_p$-universality of the generalized shifted $m$-gonal form $\Delta_{m,\mathbf a}^{(\mathbf r)}(\mathbf x)$ and the $\z_p$-universality of the diagonal quadratic form $\Delta_{4,\mathbf a}(\mathbf x)$ are equivalent for a prime $p$ with $(p,2(m-2))=1$ (resp, $(p,\frac{m-2}{2})=1$) when $m\not\eq 0\pmod{4}$ (resp, $m \eq 0\pmod{4}$).
\end{rmk}

\vskip 0.5em

\section{$\Lambda_p$-transformation}
For a generalized shifted $m$-gonal form $\Delta_{m,\mathbf a}^{(\mathbf r)}(\mathbf x)$ and an odd prime $p$ with $(p,(m-2))=1$,
we define a set
$$\Lambda_p(\Delta_{m,\mathbf a}^{(\mathbf r)}):=\{\mathbf x \in \z^n|\Delta_{m,\mathbf a}^{(\mathbf r)}(\mathbf x+\mathbf y) \eq \Delta_{m,\mathbf a}^{(\mathbf r)}(\mathbf x-\mathbf y) \pmod{p} \text{ for all } \mathbf y \in \z^n\}$$
and when $m \eq 0\pmod{4}$, we define a set
$$\Lambda_2(\Delta_{m,\mathbf a}^{(\mathbf r)}):=\{\mathbf x \in \z^n|\Delta_{m,\mathbf a}^{(\mathbf r)}(\mathbf x+\mathbf y) \eq \Delta_{m,\mathbf a}^{(\mathbf r)}(\mathbf x-\mathbf y) \pmod{8} \text{ for all } \mathbf y \in \z^n\}.$$
Then we have that
$$\Lambda_p(\Delta_{m,\mathbf a}^{(\mathbf r)})=\left\{(\ell(x_1),\cdots,\ell(x_n))|x_i \in \z\right\}$$
where $$\ell(x_i)=\begin{cases}px_i+j_i & \text{if } a_i \in \z_p^{\times} \\ x_i & \text{if } a_i \in p \z_p \end{cases}$$ 
for the unique integer $j_i$ satisfying 
$$\begin{cases}
-\frac{p}{2}+\frac{m-2-2r_i}{2(m-2)}\le j_i \le \frac{p}{2}+\frac{m-2-2r_i}{2(m-2)} \\
j_i \eq \frac{m-2-2r_i}{2(m-2)} \pmod{p}.
\end{cases}$$
Without loss of generality, let $a_1,\cdots,a_t \in \z_p^{\times}$ and $a_{t+1},\cdots,a_n \in p\z_p$, we may see that for $(\ell(x_1),\cdots,\ell(x_n)) \in \Lambda_p(\Delta^{(\mathbf r)}_{m,\mathbf a})$,
\begin{equation}\label{lambdaeq;p'}
\Delta_{m,\mathbf a}^{(\mathbf r)}(\ell(x_1),\cdots,\ell(x_n))=\sum_{i=1}^tp^2a_iP_m^{(r_i')}(x_i)+\sum_{i=t+1}^na_iP_m^{(r_i)}(x_i)+C
\end{equation}
where $r_i'=\frac{(m-2)(p+2j_i)-(m-2-2r_i)}{2p}$ 
 and $C=a_1P_m^{(r_1)}(j_1)+\cdots +a_tP_m^{(r_t)}(j_t)$.
Since $r_i'$ is an integer in $[0,m-2]$ which is relatively prime with $m-2$, the quadratic polynomial $\sum_{i=1}^tp^2a_iP_m^{(r_i')}(x_i)+\sum_{i=t+1}^na_iP_m^{(r_i)}(x_i)$ in (\ref{lambdaeq;p'}) without the constant $C$ is also a generalized shifted $m$-gonal form.
We define the generalized shifted $m$-gonal form in (\ref{lambdaeq;p'}) after 
scaling with some suitable rational integer for primitivity as
\begin{equation}\label{lambdaeq;p}
\lambda_p\left(\Delta_{m,\mathbf a}^{(\mathbf r)} \right)(\mathbf x):=\frac{1}{p^k}\left(\sum_{i=1}^tp^2a_iP_m^{(r_i')}(x_i)+\sum_{i=t+1}^na_iP_m^{(r_i)}(x_i)\right)
\end{equation}
where $k=\min\{2,\ord_p(a_{t+1}),\cdots,\ord_p(a_{n})\}$.

\vskip1em

\begin{lem} \label{lambda is regular}
Let $\Delta_{m,\mathbf a}^{(\mathbf r)}(\mathbf x)$ be a generalized primitve regular shifted $m$-gonal form. 
\begin{itemize}
\item[(1)] If $\Delta_{m,\mathbf a}^{(\mathbf r)}(\mathbf x)$ is not $\z_p$-universal for some odd prime $p$, then its $\lambda_p$-transformation
$\lambda_p (\Delta_{m,\mathbf a}^{(\mathbf r)} )(\mathbf x)$ is regular.
\item[(2)] For $m \eq 0 \pmod{4}$, if $\left<a_1,\cdots,a_n\right>\otimes \z_2 \not\supset \left<u,u'\right>\otimes \z_2$ for any $u \eq 1 \pmod{3}$ and $u' \eq 3 \pmod{4}$, then its $\lambda_2$-transformation
$\lambda_2 (\Delta_{m,\mathbf a}^{(\mathbf r)} )(\mathbf x)$ is regular.
\end{itemize}
\end{lem}
\begin{proof}
(1) Without loss of generality, let $0=\ord_p(a_1) \le \cdots \le \ord_p(a_n)$. 
It is enough to show that the generalized shifted $m$-gonal form 
$$\lambda'_p(\Delta_{m,\mathbf a}^{(\mathbf r)})(\mathbf x):=p^k\lambda_p(\Delta_{m,\mathbf a}^{(\mathbf r)})(\mathbf x)=\sum_{i=1}^tp^2a_iP_m^{(r_i')}(x_i)+\sum_{i=t+1}^na_iP_m^{(r_i)}(x_i)$$ in (\ref{lambdaeq;p}) before scaling is regular.
Note that the generalized shifted $m$-gonal form $\Delta_{m,\mathbf a}^{(\mathbf r)}(\mathbf x)=a_1P_m^{(r_1)}(x_1)+\cdots+a_nP_m^{(r_n)}(x_n)$ is written as
\begin{equation}\label{3.1}\frac{m-2}{2} \left\{a_1\left(x_1-\frac{m-2-2r_1}{2(m-2)}\right)^2+\cdots+a_n\left(x_n-\frac{m-2-2r_n}{2(m-2)}\right)^2\right\}-C\end{equation}
where $C=a_1\frac{(m-2-2r_1)^2}{8(m-2)}+\cdots+a_n\frac{(m-2-2r_n)^2}{8(m-2)} \in \z_p$.
Since $\Delta_{m,\mathbf a}^{(\mathbf r)}(\mathbf x)$ is not $\z_p$-universal, we have $t \le 2$ because otherwise (i.e., if  $0=\ord_p(a_1)=\ord_p(a_2)=\ord_p(a_3)$), the diagonal quadratic form 
$$\Delta_{4,\mathbf a} (\supset \Delta_{4,(a_1,a_2,a_3)}=\left<a_1,a_2,a_3\right>) $$
is $\z_p$-univesal, which implies that $\Delta_{m,\mathbf a}^{(\mathbf r)}(\mathbf x)$ is $\z_p$-universal by Remark \ref{rmk1}.

First we assume that $t=1$.\\
Let an integer $n$ be locally represented by $\lambda'_p(\Delta_{m,\mathbf a}^{(\mathbf r)})(\mathbf x)$.
From the equality
\begin{align*}
P_m^{(r_1)}(py+j_1) & =\frac{m-2}{2}(py+j_1)^2-\frac{m-2-2r_1}{2}(py+j_1)\\
& =p^2\left(\frac{m-2}{2}y^2+\frac{2(m-2)j_1-(m-2-2r_1)}{2p}y\right)+P_m^{(r_1)}(j_1)\\
& =p^2P_m^{(r_1')}(y)+P_m^{(r_1)}(j_1),
\end{align*}
we have that the integer $n+a_1P_m^{(r_1)}(j_1)$ is locally represented by $\Delta_{m,\mathbf a}^{(\mathbf r)}(\mathbf x)$.
From the regularity of $\Delta_{m,\mathbf a}^{(\mathbf r)}(\mathbf x)$, $n+a_1P_m^{(r_1)}(j_1)$ may be written as the $$\Delta_{m,\mathbf a}^{(\mathbf r)}(\mathbf x)=a_1P_m^{(r_1)}(x_1)+\cdots+a_nP_m^{(r_n)}(x_n)$$
for some $x_i \in \z$.
On the other hand, by an argument of congruence modulo $p$, we may obtain that $x_1 \eq j_1 \pmod{p}$, i.e., $x_1=py_1+j_1$ for some $y_1 \in \z$.
Therefore, we get that
\begin{align*}
n+a_1P_m^{(r_1)}(j_1) & =a_1P_m^{(r_1)}(py_1+j_1)+\cdots+a_nP_m^{(r_n)}(x_n) \\
 & =p^2a_1P_m^{(r_1')}(y_1)+a_1P_m^{(r_1)}(j_1)+a_2P_m^{(r_2)}(x_2)+\cdots+a_nP_m^{(r_n)}(x_n), 
\end{align*}
yielding $$n=p^2a_1P_m^{(r_1')}(y_1)+a_2P_m^{(r_2)}(x_2)+\cdots+a_nP_m^{(r_n)}(x_n).$$
Which implies that $\lambda_p'(\Delta_{m,\mathbf a}^{(\mathbf r)})(\mathbf x)$ is regular.

Next we assume that $t=2$. \\
From the assumption that $\Delta_{m,\mathbf a}^{(\mathbf r)}(\mathbf x)$ is not universal over $\z_p$, we have that the unimodular diagonal quadratic form $$\left<a_1,a_2\right>\otimes \z_p$$ is anisotropic because otherwise, 
we get a contradiction to our assumption that $\Delta_{m,\mathbf a}^{(\mathbf r)}(\mathbf x)$ is not $\z_p$-universal from Remark \ref{rmk1}. 
Similarly with the case of $t=1$, we may obtain that for an integer $n$ which is locally represneted by $\lambda'_p(\Delta_{m,\mathbf a}^{(\mathbf r)})(\mathbf x)$, the integer $n+a_1P_m^{(r_1)}(j_1)+a_2P_m^{(r_2)}(j_2) $ is written as $$a_1P_m^{(r_1)}(px_1+j_1)+a_2P_m^{(r_2)}(px_2+j_2)+a_3P_m^{(r_3)}(x_3)+\cdots+a_nP_m^{(r_n)}(x_n),$$
for some $x_i \in \z$, yielding 
$n$ is written as
$$n=p^2a_1P_m^{(r_1')}(x_1)+p^2a_2P_m^{(r_2')}(x_2)+a_3P_m^{(r_3)}(x_3)+\cdots+a_nP_m^{(r_n)}(x_n).$$

(2) Without loss of generality, we assume that $a_1,\cdots, a_t \in \z_2^{\times}$ and $a_{t+1},\cdots, a_n \in 2\z_2$ where $1 \le t \le 3$.
We may have that when $t=2$, $$a_1\eq a_2 \pmod{4} \ \text{ and } \ a_{3},\cdots, a_n \in 4\z_2$$
and when $t=3$, $$a_1\eq a_2 \eq a_3 \pmod{4} \ \text{ and } \ a_{4},\cdots, a_n \in 4 \z_2$$
since $\left<a_1,\cdots,a_n\right> \otimes \z_2 \not\supset \left<u,u'\right> \otimes \z_2$ for any $u \eq 1 \pmod{4}$ and $u' \eq 3 \pmod{4}$.
One may show that the shifted $m$-gonal form $\lambda_2(\Delta_{m,\mathbf a}^{(\mathbf r)})(\mathbf x)$ is regular case by case on the number $t$ of units of $\z_2$ in $\{a_1,\cdots, a_n\}$ similarly with the above (1).
We omit the proof in this paper.
\end{proof}

\vskip 1em

For a generalized primitive shifted $m$-gonal form $\Delta_{m,\mathbf a}^{(\mathbf r)}(\mathbf x)$ of rank $n\ge4$, for an odd prime $p$ with $(p,m-2)=1$, after finite times $\lambda_p$-transformations, we may obtain a generalized primitive regular shifted $m$-gonal form
\begin{equation}\label{lambdas}\lambda_p\circ\cdots \circ \lambda_p(\Delta_{m,\mathbf a}^{(\mathbf r)})(\mathbf x):=a_1'P_m^{(r_1')}(x_1)+\cdots+a_n'P_m^{(r_n')}(x_n)\end{equation} which is $\z_p$-universal with $\z_p$-universal diagonal quadratic form $\left<a_1',\cdots, a_n'\right>$ by Lemma \ref{lambda is regular} (1).
We define {\it $\widetilde{\lambda_p}$-transformation} of $\Delta_{m,\mathbf a}^{(\mathbf r)}(\mathbf x)$ as the generalized primitive regular shifted regular $m$-gonal form of (\ref{lambdas}) 
\begin{equation}\label{Lambda}\widetilde{\lambda_p}(\Delta_{m,\mathbf a}^{(\mathbf r)})(\mathbf x):=\lambda_p\circ\cdots \circ \lambda_p(\Delta_{m,\mathbf a}^{(\mathbf r)})(\mathbf x)\end{equation}
which is $\z_p$-universal.

For a generalized primitive shifted $m$-gonal form $\Delta_{m,\mathbf a}^{(\mathbf r)}(\mathbf x)$ of rank $n\ge4$ with $m \eq 0 \pmod{4}$, after finite times $\lambda_2$-transformations, we may obtain a generalized primitive regular shifted $m$-gonal form
\begin{equation}\label{lambdas2}\lambda_2\circ\cdots \circ \lambda_2(\Delta_{m,\mathbf a}^{(\mathbf r)})(\mathbf x):=a_1'P_m^{(r_1')}(x_1)+\cdots+a_n'P_m^{(r_n')}(x_n)\end{equation} which is 
\begin{equation}\label{lambda2 ex'qf''}
\text{$\z_2$-universal with the $\z_2$-universal $\left<a_1',\cdots, a_n'\right>\otimes \z_2$}
\end{equation}
 or 
\begin{equation}\label{lambda2 ex'qf'}
\text{not $\z_2$-universal with $\left<a_1',\cdots,a_n'\right>\otimes \z_2 \supset \left<u,u'\right>\otimes \z_2$}
\end{equation}  
for some $u \eq 1\pmod{4}$ and $u' \eq 3 \pmod{4}$ by Lemma \ref{lambda is regular} (2).
We define {\it $\widetilde{\lambda_2}$-transformation} of $\Delta_{m,\mathbf a}^{(\mathbf r)}(\mathbf x)$ as the generalized primitive regular shifted regular $m$-gonal form of (\ref{lambdas2}) 
\begin{equation}\label{Lambda}\widetilde{\lambda_2}(\Delta_{m,\mathbf a}^{(\mathbf r)})(\mathbf x):=\lambda_2\circ\cdots \circ \lambda_2(\Delta_{m,\mathbf a}^{(\mathbf r)})(\mathbf x)\end{equation}
which is (\ref{lambda2 ex'qf''}) or (\ref{lambda2 ex'qf'}).

\vskip 1em

\begin{rmk}\label{rmk 2rep}
For a generalized primitive regular shifted $m$-gonal form $\Delta_{m,\mathbf a}^{(\mathbf r)}(\mathbf x)$ of rank $n\ge4$, let its $\widetilde{\lambda_2}$-transformation $\widetilde{\lambda_2}(\Delta_{m,\mathbf a}^{(\mathbf r)})(\mathbf x):=\Delta_{m,\mathbf a'}^{(\mathbf r')}(\mathbf x)$
is not universal over $\z_2$, i.e., the diagonal quadratic form $\left<a_1'\cdots,a_n'\right>\otimes \z_2$ is not universal with 
\begin{equation}\label{lambda2 ex'qf}
\left<a_1',\cdots,a_n'\right> \otimes \z_2 \supset \left<u,u'\right>\otimes \z_2
\end{equation}
for some $u \eq 1 \pmod{4}$ and $u' \eq 3 \pmod{4}$.
Then the generalized primitive regular shifted $m$-gonal form $\widetilde{\lambda_2}(\Delta_{m,\mathbf a}^{(\mathbf r)})(\mathbf x)=\Delta_{m,\mathbf a'}^{(\mathbf r')}(\mathbf x)$ would represent either every non-negative odd integer or every non-negative even integer over $\z_2$.
For with
$$\widetilde{\lambda_2}(\Delta_{m,\mathbf a}^{(\mathbf r)})(\mathbf x)=\Delta_{m,\mathbf a'}^{(\mathbf r')}(\mathbf x)=\frac{m-2}{2}\Delta_{4,\mathbf a'}(\mathbf x')-C$$
where $\mathbf x'=\mathbf x-\left(\frac{m-2-2r_1'}{2(m-2)},\cdots ,\frac{m-2-2r_n'}{2(m-2)}\right)$ and $C=\sum_{i=1}^n a_i'\frac{(m-2-2r_i')^2}{8(m-2)} \in \z_2$, the diagonal quadratic form $\Delta_{4,\mathbf a'}\otimes \z_2$ with (\ref{lambda2 ex'qf})
represents every unit of $\z_2$.
Therefore the generalized primitive shifted $m$-gonal form $\widetilde{\lambda_2}(\Delta_{m,\mathbf a}^{(\mathbf r)})$ would represent every odd integer (or resp, even integer) if $C\in 2\z_2$ (or resp, if $C \in \z_2^{\times}$) over $\z_2$.
Plus, the diagonal quadratic form $\Delta_{4,\mathbf a'}\otimes \z_2$ represents not only units of $\z_2$ but also every non-unit integer in $(a_1'+a_2')+8\z_2$ and $(a_1'+a_2'+a_3')+8\z_2$.
Hence we may induce that the generalized primitive shifted $m$-gonal form $\widetilde{\lambda_2}(\Delta_{m,\mathbf a}^{(\mathbf r)})$ represents every non-negative integer in at least five cosets of $8\z$ over $\z_2$ (among them, four cosets would be in a same coset of $2\z$) and if $\ord_2(a_3)<3$, then every non-negative integer in at least six cosets of $8\z$ 
over $\z_2$. 
\end{rmk}

\vskip 1em

\begin{lem} \label{other q}
Let $\Delta_{m,\mathbf a}^{(\mathbf r)}(\mathbf x)$ be a generalized primitive shifted $m$-gonal form.
\begin{itemize}
\item[(1)]The $\z_q$-universality of $\Delta_{m,\mathbf a}^{(\mathbf r)}(\mathbf x)$ and $\lambda_p(\Delta_{m,\mathbf a}^{(\mathbf r)})(\mathbf x)$ coincide for other prime $q \not=p$.
\item[(2)]If $\Delta_{m,\mathbf a}^{(\mathbf r)}(\mathbf x)$ is not $\z_2$-universal satisfying (\ref{lambda2 ex'qf}) (resp, not satisfying (\ref{lambda2 ex'qf})), then its $\lambda_p$ transformation 
$\lambda_p(\Delta_{m,\mathbf a}^{(\mathbf r)})(\mathbf x)$ is also not $\z_2$-universal satisfying (\ref{lambda2 ex'qf}) (resp, not satisfying (\ref{lambda2 ex'qf})) for an odd prime $p$.
\end{itemize}
\end{lem}
\begin{proof}
By Proposition \ref{prop1}, it would be enough to show that for primes $q$ with 
\begin{equation}\label{qqq}
\text{$(q,2(m-2))=1$ (resp, $\left(q,\frac{(m-2)}{2}\right)=1$)}
\end{equation}
when $m \not\eq 0\pmod{4}$ (resp, $m \eq 0\pmod{4}$).
For $\lambda_p(\Delta_{m,\mathbf a}^{(\mathbf r)})(\mathbf x):=\Delta_{m,\mathbf a'}^{(\mathbf r')}(\mathbf x)$, we may see that the diagonal quadratic form 
$\Delta_{4,\mathbf a'}$ is isometric with $$\Delta_{4,\mathbf a} \ \text{ or } \ \Delta_{4,p\mathbf a'}$$ over $\z_q$, so the universality of $\Delta_{4,\mathbf a'}$ over $\z_q$ and the universality of $\Delta_{4,\mathbf a}$ over $\z_q$ would coincide.
This completes the claim of (1) by Remark \ref{rmk1} and (2) directly.
\end{proof}

\vskip 1em

\begin{rmk} \label{rmk Lambda}
For a generalized primitive regular shifted $m$-gonal form $\Delta_{m,\mathbf a}^{(\mathbf r)}(\mathbf x)$ of rank $n\ge4$, let $T:=\{p_1,\cdots,p_s\}$ be a set of primes $p$ for which $\Delta_{m,\mathbf a}^{(\mathbf r)}(\mathbf x)$ is not $\z_p$-universal (or resp, not $\z_2$-universal not satsifying (\ref{lambda2 ex'qf})) when $p$ is an odd prime (or resp, $p=2$).
Then $T$ would be a finite set of primes $p$ with $p \nmid 2(m-2)$ (resp, $p \nmid \frac{(m-2)}{2}$) when $m \not\eq 0 \pmod{4}$ (resp, $m \eq 0 \pmod{4}$) by Proposition \ref{prop1} and Remark \ref{rmk1}.
By Lemma \ref{lambda is regular} and Lemma \ref{other q}, we may see that
$$\lambda(\Delta_{m,\mathbf a}^{(\mathbf r)})(\mathbf x):=\widetilde{\lambda_{p_{s}}}\circ\cdots\circ\widetilde{\lambda_{p_{2}}}\circ\widetilde{\lambda_{p_{1}}}(\Delta_{m,\mathbf a}^{(\mathbf r)})(\mathbf x):=\Delta_{m,\mathbf a'}^{(\mathbf r')}(\mathbf x)$$
is a generalized primitive regular shifted $m$-gonal form which is $\z_p$-universal for all odd primes $p$ and $\z_2$-universal or not $\z_2$-universal satisfying (\ref{lambda2 ex'qf}).
\end{rmk}

\vskip 1em

\begin{thm} \label{fin rank}
The rank $n$ of a generalized regular shifted $m$-gonal form $\Delta_{m,\mathbf a}^{(\mathbf r)}(\mathbf x)$ satisfies 
$$ n \ge 
\begin{cases} \log_2  \floor{\frac{m}{2}}  &  \text{ when } m \not\eq 0 \pmod{4}\\
\log_2 \left( \frac{1}{2}\cdot\floor{\frac{m}{2}}+\floor{\frac{1}{8}\cdot\floor{\frac{m}{2}}} \right) &  \text{ when } m  \eq 0 \pmod{4}.
\end{cases}$$
\end{thm}
\begin{proof}
Without loss of generality, we may assume that $\Delta_{m,\mathbf a}^{(\mathbf r)}(\mathbf x)$ is primitive by scaling if it is necessary.
For a generalized primitive regular shifted $m$-gonal form $\Delta_{m,\mathbf a}^{(\mathbf r)}(\mathbf x)$ of rank $n \ge 4$, its $\lambda$-transformation $\lambda(\Delta_{m,\mathbf a}^{(\mathbf r)})(\mathbf x)$ in Remark \ref{rmk Lambda} locally represents every non-negative integer (resp, every non-negative integer in at least five cosets of $8\z$ (among them, four cosets would be in a same coset of $2\z$)) when $m \not\eq 0 \pmod{4}$ (resp, $m \eq 0 \pmod{4}$) by Remark \ref{rmk 2rep}.
Therefore the generalized primitive regular shifted $m$-gonal form $\lambda(\Delta_{m,\mathbf a}^{(\mathbf r)})(\mathbf x)$ would locally represent at least $\floor{\frac{m}{2}}$ (resp, $\frac{1}{2}\cdot\floor{\frac{m}{2}}+\floor{\frac{1}{8}\cdot\floor{\frac{m}{2}}}$) integers among the $\floor{\frac{m}{2}}$ non-negative integers in $\left[0,\frac{m-2}{2}\right]$ when $m \not\eq 0 \pmod{4}$ (resp, $m \eq 0 \pmod{4}$) by Remark \ref{rmk 2rep}.
On the other hand, the number of integers which are globally represented by $\lambda(\Delta_{m,\mathbf a})(\mathbf x)$ in $\left[0,\frac{m-2}{2}\right]$ could not exceed $2^n$ because there are only two shifted $m$-gonal numbers of level $r$ in which are less than or equal to $\frac{m-2}{2}$ for each fixed $(m;r)$.
Hence we obtain that
$$2^n \ge 
\begin{cases}\floor{\frac{m}{2}}  &  \text{ when } m \not\eq 0 \pmod{4}\\
\frac{1}{2}\cdot\floor{\frac{m}{2}}+\floor{\frac{1}{8}\cdot\floor{\frac{m}{2}}}  &  \text{ when } m  \eq 0 \pmod{4}
\end{cases}$$
for the regularity of $\lambda(\Delta_{m,\mathbf a}^{(\mathbf r)})(\mathbf x)$.
This completes the proof.
\end{proof}

\begin{rmk}
Theorem \ref{fin rank} says that for each fixed $n\ge 4$, there are only finitely many $m$ for which there is a generalized regular shifted $m$-gonal form of rank $n$.
\end{rmk}

\begin{thm} \label{uni min}
The minimal rank for generalized universal $m$-gonal form $\Delta_{m,\mathbf a}(\mathbf x)$ is $\ceil{\log_2(m+2)}$. 
\end{thm}
\begin{proof}
See \cite{P}.
\end{proof}

\begin{rmk} \label{rank rmk1}
Since a generalized universal $m$-gonal form is regular, the minimal rank for generalized regular $m$-gonal form would be certainly less than or equal to $\ceil{\log_2(m+2)}$ by Theorem \ref{fin rank}.
But the Theorem \ref{fin rank} leaves a small gap for the indeed minimal rank for generalized regular $m$-gonal form.
\end{rmk}

\vskip 1em
\begin{lem} \label{lem33}
\begin{itemize}
\item[(1)] For $m\ge14$ with $m \not\equiv 0 \pmod{4}$, a generalized primitive regular shifted $m$-gonal form $\Delta_{m,\mathbf a}^{(\mathbf r)}(\mathbf x)$ of rank $n\ge4$ is $\z_p$-universal for every prime $p \not=3$.
\item[(2)] For $m\ge28$ with $m \equiv 0 \pmod{4}$, a generalized primitive regular shifted $m$-gonal form $\Delta_{m,\mathbf a}^{(\mathbf r)}(\mathbf x)$ of rank $n \ge 4$ is $\z_p$-universal for every prime $p \not=2,3$.
\end{itemize}
\end{lem}
\begin{proof}
(1) Let $T=\{p_1,\cdots, p_s\}$ be a set of all primes $p$ for which the generalized primitive regular shifted $m$-gonal form $\Delta_{m,\mathbf a}^{(\mathbf r)}(\mathbf x)$ is not $\z_p$-universal.
Then $T$ would be a finite set of odd primes $p$ with $p \nmid m-2$ by Proposition \ref{prop1} and Remark \ref{rmk1}.
If $T\subseteq \{3\}$, then we are done.
Now we assume that $T \not\subseteq \{3\}$.
Throughout this proof, for $p:=p_s\not=3$, we put 
$$\widetilde{\lambda}(\Delta_{m,\mathbf a}^{(\mathbf r)})(\mathbf x):=\widetilde{\lambda_{p_{s-1}}}\circ\cdots\circ\widetilde{\lambda_{p_{2}}}\circ\widetilde{\lambda_{p_{1}}}(\Delta_{m,\mathbf a}^{(\mathbf r)})(\mathbf x):=\Delta_{m,\mathbf a'}^{(\mathbf r')}(\mathbf x)$$
where $0=\ord_p(a_1')\le \cdots \le \ord_p(a_n')$.
Then the generalized primitive regular shifted $m$-gonal form $\widetilde{\lambda}(\Delta_{m,\mathbf a}^{(\mathbf r)})(\mathbf x)$ would be
$$\begin{cases}
\text{$\z_q$-universal} & \text{for all primes $q\not=p$} \\
\text{not $\z_p$-universal}. & \text{} \\
\end{cases}$$
By  Remark \ref{rmk1}, we have that $1 \le \ord_p(a_3')\le \cdots \le \ord_p(a_n')$.
When $\ord_p(a_2')=0$ (resp, $\ord_p(a_2')>0$), the generalized shifted $m$-gonal form $\widetilde{\lambda}(\Delta_{m,\mathbf a}^{(\mathbf r)})(\mathbf x):=\Delta_{m,\mathbf a'}^{(\mathbf r')}(\mathbf x)$ locally represents every non-negative integer in $p-1$ (resp, $\frac{p-1}{2}$) cosets of $p\z$ because the generalized shifted $m$-gonal form $\widetilde{\lambda}(\Delta_{m,\mathbf a}^{(\mathbf r)})(\mathbf x)$ is $\z_q$-universal for every $q \not=p$ and the diagonal quadratic form $\Delta_{4,\mathbf a'} \otimes \z_p$ represents every integer in $p-1$ (resp, $\frac{p-1}{2}$) cosets of $p\z$.

First we show that $p_s < \frac{m}{2}$.
For a contradiction, let $p_s:=p \ge \frac{m}{2}$.
The generalized shifted $m$-gonal form $\widetilde{\lambda}(\Delta_{m,\mathbf a}^{(\mathbf r)})(\mathbf x)$ would locally represent at least $$p-1\ \left(\text{resp, }\frac{p-1}{2} \right)$$
(more precisely, $p-1$ or $p$ (resp, $\frac{p-1}{2}$ or $\frac{p+1}{2}$)) integers among the $p$ non-negative integers in $[0,p-1]$ when $\ord_p(a_2')=0 \ (\text{resp, }\ord_p(a_2')>0)$ because the generalized shifted $m$-gonal form $\widetilde{\lambda}(\Delta_{m,\mathbf a}^{(\mathbf r)})(\mathbf x)$ locally represents every non-negative integer in $p-1$ (resp, $\frac{p-1}{2}$) cosets of $p\z$.
On the other hand, the number of integers in $[0,p-1]$ which are globally represented by $\widetilde{\lambda}(\Delta_{m,\mathbf a}^{(\mathbf r)})(\mathbf x)$ could not exceed
$$\prod\limits_{i=1}^22\sqrt{\frac{2}{m-2}\left(\frac{p}{a_i}+\frac{(m-2-2r_i)^2}{8(m-2)}\right)}\left(\text{resp, }2\sqrt{\frac{2}{m-2}\left(\frac{p}{a_1}+\frac{(m-2-2r_1)^2}{8(m-2)}\right)} \right)$$ 
(which is less than $\left(2\sqrt{\frac{2p}{m-2}+\frac{1}{4}}\right)^2=\frac{8p}{m-2}+1\ \left(\text{resp, }2\sqrt{\frac{2p}{m-2}+\frac{1}{4}} \right)$) becuase $p \le a_3',\cdots, a_n'\ (\text{resp, }p \le a_2',\cdots, a_n'$) when $\ord_p(a_2')=0 \ (\text{resp, }\ord_p(a_2')>0).$
Which leads to a contradiction to the regularity of $\widetilde{\lambda}(\Delta_{m,\mathbf a}^{(\mathbf r)})(\mathbf x)$ for 
$$\frac{8p}{m-2}+1<p-1 \left(\text{resp, }2\sqrt{\frac{2p}{m-2}+\frac{1}{4}}<\frac{p-1}{2} \right)$$
when $\ord_p(a_2')=0 \ (\text{resp, }\ord_p(a_2')>0).$

Next we show that $p_s < 7$.
For a contradiction, let $7 \le p_s:= p < \frac{m}{2}$.
Similarly with the above, the generalized shifted $m$-gonal form $\widetilde{\lambda}(\Delta_{m,\mathbf a}^{(\mathbf r)})(\mathbf x)$ would locally represent at least
$$p-1\ \left(\text{resp, }\frac{p-1}{2} \right)$$
(more precisely, $p-1$ or $p$ (resp, $\frac{p-1}{2}$ or $\frac{p+1}{2}$)) integers among the $p$ non-negative integers in $[0,p-1]$ when $\ord_p(a_2)=0 \ (\text{resp, }\ord_p(a_2)>0)$ because the generalized shifted $m$-gonal form $\widetilde{\lambda}(\Delta_{m,\mathbf a}^{(\mathbf r)})(\mathbf x)$ locally represents every non-negative integer in $p-1$ (resp, $\frac{p-1}{2}$) cosets of $p\z$. 
On the other hand, the number of integers in $[0,p-1\left(<\frac{m-2}{2}\right)]$ which are globally represented by $\widetilde{\lambda}(\Delta_{m,\mathbf a}^{(\mathbf r)})(\mathbf x)$ would be at most $$2\cdot 2=4\ (\text{resp, }2).$$
 becuase $p \le a_3',\cdots, a_n'\ (\text{resp, }p \le a_2',\cdots, a_n'$) when $\ord_p(a_2)=0 \ (\text{resp, }\ord_p(a_2)>0).$
Which leads to a contradiction to the regularity of $\widetilde{\lambda}(\Delta_{m,\mathbf a}^{(\mathbf r)})(\mathbf x)$ for 
$$4 <p-1 \left(\text{resp, }2 <\frac{p-1}{2} \right)$$
when $\ord_p(a_2)=0 \ (\text{resp, }\ord_p(a_2)>0).$

Finally we show that $5 \not \in T$.
Suppose that $p_s:=5 \in T$.

\begin{itemize}
\item[Case 1.] $\ord_5(a_2')=0$.\\
Since the generalized shifted $m$-gonal form 
$\widetilde{\lambda}(\Delta_{m,\mathbf a}^{(\mathbf r)})(\mathbf x)=\Delta_{m,\mathbf a'}^{(\mathbf r')}(\mathbf x)$
is not $\z_5$-universal, the diagonal quadratic form $\left<a_1',\cdots,a_n'\right>=\Delta_{4,\mathbf a'}$ is also not $\z_5$-universal  by Remark \ref{rmk1},
so we have that $$a_1' \not\eq \pm a_2' \pmod{5}$$ and $$1<\ord_5(a_4')\le \cdots \le \ord_5(a_n').$$
Since the generalized shifted $m$-gonal form $\widetilde{\lambda}(\Delta_{m,\mathbf a}^{(\mathbf r)})(\mathbf x)$ locally represents at least
$\frac{4}{5}\cdot 5 =4$ integers in $[0,4]$, for the regularity of $\widetilde{\lambda}(\Delta_{m,\mathbf a}^{(\mathbf r)})(\mathbf x)$, we have that 
\begin{equation}\label{a1'a2'r1r2}(a_1',a_2') \in \{(1,2),(1,3)\} \quad \text{ and }\quad  r_1',r_2' \in \{1,m-3\}.\end{equation}
While the generalized shifted $m$-gonal form $\widetilde{\lambda}(\Delta_{m,\mathbf a}^{(\mathbf r)})(\mathbf x)=\Delta_{m,\mathbf a'}^{(\mathbf r')}(\mathbf x)$ with $\ord_5(a_1')=\ord_5(a_2')=0$ locally represents at least
$\frac{4}{5}\cdot 10 =8$ integers in $[0,9]$, 
the generalized shifted $m$-gonal form $a_1'P_m^{(r_1')}(x_1)+a_2'P_m^{(r_2')}(x_2)$ with (\ref{a1'a2'r1r2}) (globally) represents only four integers in $[0,9]$.
For the regularity of $\widetilde{\lambda}(\Delta_{m,\mathbf a}^{(\mathbf r)})(\mathbf x)$ with $a_3' \eq \cdots \eq a_n' \eq 0 \pmod{5}$, we have that 
\begin{equation}\label{a3'r3}a_3'=5  \quad \text{ and }\quad r_3' \in \{1,m-3\}.\end{equation}
And then we get a contradiction to the regularity of $\widetilde{\lambda}(\Delta_{m,\mathbf a}^{(\mathbf r)})(\mathbf x)$ because the generalized shifted $m$-gonal form $\widetilde{\lambda}(\Delta_{m,\mathbf a}^{(\mathbf r)})(\mathbf x)$ under (\ref{a1'a2'r1r2}) and (\ref{a3'r3}) with $25 \le a_4',\cdots, a_n'$ could not represent $12$ different non-negative integers in $[0,14]$
while $\widetilde{\lambda}(\Delta_{m,\mathbf a}^{(\mathbf r)})(\mathbf x)$ locally represents at least
$\frac{4}{5}\cdot 15 =12$ integers in $[0,14]$.
 
\item[Case 2.] $\ord_5(a_2')>0$.\\
If $\ord_5(a_4')=1$ or $(\ord_5(a_3'),\ord_5(a_4'))=(1,2)$, then the generalized shifted $m$-gonal form $\widetilde{\lambda}(\Delta_{m,\mathbf a}^{(\mathbf r)})(\mathbf x)$ locally represents $\frac{3}{5}\cdot 5=3$ different integers in $\left[0,4\left(<\frac{m-2}{2}\right)\right]$ because $\widetilde{\lambda}(\Delta_{m,\mathbf a}^{(\mathbf r)})(\mathbf x)$ is $\z_q$-universal for every $q \not=5$ and the diagonal quadratic form $\left<a_1',\cdots,a_n'\right>\otimes \z_5$ represents every integer in three cosets of $5\z$.
For the regularity of $\widetilde{\lambda}(\Delta_{m,\mathbf a}^{(\mathbf r)})(\mathbf x)$, we have that $\ord_5(a_4')>1$ and $(\ord_5(a_3'),\ord_5(a_4')) \not=(1,2)$.
Since $\widetilde{\lambda}(\Delta_{m,\mathbf a}^{(\mathbf r)})(\mathbf x)$ locally represents at least $\frac{2}{5}\cdot 5=2$ integers among the $5$ integers in $\left[0,4 \left(<\frac{m-2}{2}\right)\right]$, for the regularity of $\widetilde{\lambda}(\Delta_{m,\mathbf a}^{(\mathbf r)})(\mathbf x)$ with $a_2' \eq \cdots \eq a_n' \eq 0 \pmod{5}$,
we have that 
\begin{equation}\label{T5}a_1'P_m^{(r_1')}\left(\sgn\left(\frac{m-2}{2}-r_1'\right)\cdot 1\right ) \le 4.\end{equation}
Since $\widetilde{\lambda}(\Delta_{m,\mathbf a}^{(\mathbf r)})(\mathbf x)=\Delta_{m,\mathbf a'}^{(\mathbf r')}(\mathbf x)$ with $a_1' \not\eq 0 \pmod{5}$ locally represents at least $\frac{2}{5}\cdot 10=4$ different integers among the $10$ integers in $[0,9(<m)]$, for the regularity of $\widetilde{\lambda}(\Delta_{m,\mathbf a}^{(\mathbf r)})(\mathbf x)$, we have that 
\begin{equation}\label{a3'r3''}a_2'=5  \quad \text{ and }\quad r_2' \in \{1,m-3\}.\end{equation}
And then the generalized shifted $m$-gonal form $\widetilde{\lambda}(\Delta_{m,\mathbf a}^{(\mathbf r)})(\mathbf x)$ locally represents at least $\left( \frac{2}{5}+\frac{2}{25}\right) \cdot 25=12$ different integers among the $25$ integers in $[0,24]$.
Note that there are at most three different generalized shifted $m$-gonal numbers of level $r_1'$ except $P_m^{(r_1')}(0)=0$ which are less than or equal to $24 \left(\le 2m-4=\frac{P_m(-2)+P_m(2)}{2}\right)$ and only one generalized shifted $m$-gonal number of level $r_2'$ except $P_m^{(r_2')}(0)=0$ which are less than $5=\frac{25}{a_2'}$.
Under (\ref{T5}) and (\ref{a3'r3''}), for the regularity of $\widetilde{\lambda}(\Delta_{m,\mathbf a}^{(\mathbf r)})(\mathbf x)$ with $a_3' \eq \cdots \eq a_n' \eq 0 \pmod{5}$, we have that $a_3'<25$, i.e., $\ord_5(a_3')=1$.
Since the number of different residues of integers in $[0,9]$ which are locally represented by $\widetilde{\lambda}(\Delta_{m,\mathbf a}^{(\mathbf r)})(\mathbf x)$ with $\ord_5(a_1')=0$ and $\ord_5(a_2')=\ord_5(a_3')=1$ would be $3$, we obtain that 
\begin{equation}\label{T5'} 0, \ a_1'P_m^{(r_1')}(1), \ a_1'P_m^{(r_1')}(-1)<10\end{equation}
have all different residues modulo $5$.
But there is no $(m;a_1',r_1')$ with $m \ge 14$ and $(m-2,r_1')=1$ satisfying (\ref{T5}) and $(\ref{T5'})$.
So we complete to show that $T \subseteq \{3\}$.
\end{itemize}

(2) Let $T:=\{p_1,\cdots, p_s\}$ be a set of all primes $p$ for which the generalized shifted $m$-gonal form $\Delta_{m,\mathbf a}^{(\mathbf r)}(\mathbf x)$ is not $\z_p$-universal (or resp, not $\z_2$-universal not satsifying (\ref{lambda2 ex'qf})) when $p$ is an odd prime (or resp, $p=2$).
Then $T$ would be a finite set of primes $p$ with $p \nmid \frac{m-2}{2}$ by Proposition \ref{prop1} and Remark \ref{rmk1}.
If $T \subseteq \{2,3\}$, then we are done.
Now we assume that $T \not\subseteq \{2,3\}$.
Throughout this proof, for $p:=p_s \not\in \{2,3\}$, we put 
$$\widetilde{\lambda}(\Delta_{m,\mathbf a}^{(\mathbf r)})(\mathbf x):=\widetilde{\lambda_{p_{s-1}}}\circ\cdots\circ\widetilde{\lambda_{p_{2}}}\circ\widetilde{\lambda_{p_{1}}}(\Delta_{m,\mathbf a}^{(\mathbf r)})(\mathbf x):=\Delta_{m,\mathbf a'}^{(\mathbf r')}(\mathbf x)$$
where $0=\ord_p(a_1')\le \cdots \le \ord_p(a_n')$. 
Then the generalized primitive regular shifted $m$-gonal form $\widetilde{\lambda}(\Delta_{m,\mathbf a}^{(\mathbf r)})(\mathbf x)$ would be 
$$\begin{cases}
\text{$\z_q$-universal} & \text{for all primes $q\not=2, p$} \\
\text{$\z_2$-universal or not $\z_2$-universal satisfying (\ref{lambda2 ex'qf})} & \text{} \\
\text{not $\z_p$-universal}. & \text{} \\
\end{cases}$$
Recall that the generalized shifted $m$-gonal form $\widetilde{\lambda}(\Delta_{m,\mathbf a}^{(\mathbf r)})(\mathbf x)$ represents every non-negative integer in at least five cosets of $8\z$ (among them, four cosets of $8\z$ would be in a coset of $2\z$, i.e., every non-negative odd integer or every non-neagative even integer) over $\z_2$.
When $\ord_p(a_2')=0$ (resp, $\ord_p(a_2')>0$), the generalized shifted $m$-gonal form $\widetilde{\lambda}(\Delta_{m,\mathbf a}^{(\mathbf r)})(\mathbf x)$ represents every non-negative integer in at least $p-1$ cosets (resp, $\frac{p-1}{2}$ cosets) of $p\z$ over $\z_p$ becasuse so does the diagonal quadratic form $\Delta_{4,\mathbf a'}(\mathbf x)$.

First we show that $p_s< \frac{m+1}{2}$.
For a contradiction, let $p_s:=p \ge \frac{m+1}{2}$.
The generalized shifted $m$-gonal form $\widetilde{\lambda}(\Delta_{m,\mathbf a}^{(\mathbf r)})(\mathbf x)$ locally represents at least 
$$p-1\ \left(\text{resp, }\frac{p-1}{2} \right)$$
 integers having all different residues modulo $p$ among the $2p$ non-negative integers in $[0,2p-1]$
when $\ord_p(a_2')=0 \ (\text{resp, }\ord_p(a_2')>0)$ because $\widetilde{\lambda}(\Delta_{m,\mathbf a}^{(\mathbf r)})(\mathbf x)$ is $\z_q$-universal for every $q \not=2,p$, represents every odd integer or every even integer over $\z_2$, and represents every integer in $p-1$ (resp, $\frac{p-1}{2}$) cosets of $p\z$ over $\z_p$.
On the other hand, the number of different residues modulo $p$ of integers in $[0,2p-1]$ which are globally represented by $\widetilde{\lambda}(\Delta_{m,\mathbf a}^{(\mathbf r)})(\mathbf x)$ would be at most
$$\prod\limits_{i=1}^22\sqrt{\frac{2}{m-2}\left(\frac{2p}{a_i'}+\frac{(m-2-2r_i)^2}{8(m-2)}\right)}\left(\text{resp, }2\sqrt{\frac{2}{m-2}\left(\frac{2p}{a_1'}+\frac{(m-2-2r_1)^2}{8(m-2)}\right)} \right)$$ 
(which is less than $\left(2\sqrt{\frac{4p}{m-2}+\frac{1}{4}}\right)^2=\frac{16p}{m-2}+1\ (\text{resp, }2\sqrt{\frac{4p}{m-2}+\frac{1}{4}}$)) because $ a_3' \eq \cdots \eq a_n' \eq 0 \pmod{p}\ (\text{resp, }a_2' \eq \cdots \eq a_n' \eq 0 \pmod{p}$) when $\ord_p(a_2')=0 \ (\text{resp, }\ord_p(a_2')>0).$
Which leads to a contradiction to the regularity of $\widetilde{\lambda}(\Delta_{m,\mathbf a}^{(\mathbf r)})(\mathbf x)$ for 
$$\frac{16p}{m-2}+1<p-1 \  \left(\text{resp, }2\sqrt{\frac{4p}{m-2}+\frac{1}{4}} <\frac{p-1}{2}  \right)$$
when $\ord_p(a_2)=0 \ (\text{resp, }\ord_p(a_2)>0)$.

Next we show that $p_s < 11$.
For a contradiction, let $11 \le p_s:= p < \frac{m+1}{2}$.
Similarly with the above, we may see that  the generalized shifted $m$-gonal form $\widetilde{\lambda}(\Delta_{m,\mathbf a}^{(\mathbf r)})(\mathbf x)$ locally represents at least
$$p-1\ \left(\text{resp, }\frac{p-1}{2} \right)$$
integers having all different residues modulo $p$ among the $2p$ non-negative integers in $[0,2p-1]$
when $\ord_p(a_2')=0 \ (\text{resp, }\ord_p(a_2')>0)$.
On the other hand, the number of different residues modulo $p$ of integers in $[0,2p-1\left(< m \right)]$ which are globally represented by $\widetilde{\lambda}(\Delta_{m,\mathbf a}^{(\mathbf r)})(\mathbf x)$ would be at most $$3\cdot 3=9\ (\text{resp, }3).$$
 becuase $ a_3' \eq \cdots \eq a_n' \eq 0 \pmod{p}\ (\text{resp, }a_2' \eq \cdots \eq a_n' \eq 0 \pmod{p}$) when $\ord_p(a_2')=0 \ (\text{resp, }\ord_p(a_2')>0).$
Which leads to a contradiction to the regularity of $\widetilde{\Lambda}(\Delta_{m,\mathbf a}^{(\mathbf r)})(\mathbf x)$ for 
$$9 <p-1 \left(\text{resp, }3 <\frac{p-1}{2} \right)$$
when $\ord_p(a_2)=0 \ (\text{resp, }\ord_p(a_2)>0).$

Now we show that $7 \not\in T$.
For a contradiction, let $7=p_s$.
The generalized shifted $m$-gonal form $\widetilde{\lambda}(\Delta_{m,\mathbf a}^{(\mathbf r)})(\mathbf x)$ locally represents at least
$$6\ \left(\text{resp, } 3 \right)$$
integers having all different residues modulo $7$ among the $14$ non-negative integers in $[0,13]$
when $\ord_7(a_2')=0 \ (\text{resp, }\ord_7(a_2')>0)$.
On the other hand, the number of residues modulo $7$ of integers in $\left[0,13\left(\le \frac{m-2}{2} \right)\right]$ which are globally represented by $\widetilde{\lambda}(\Delta_{m,\mathbf a}^{(\mathbf r)})(\mathbf x)$ would be at most $$2\cdot 2=4\ (\text{resp, }2).$$
 because $ a_3' \eq \cdots \eq a_n' \eq 0 \pmod{7}\ (\text{resp, }a_2' \eq \cdots \eq a_n' \eq 0 \pmod{7}$) when $\ord_7(a_2')=0 \ (\text{resp, }\ord_7(a_2')>0).$
Which leads to a contradiction to the regularity of $\widetilde{\lambda}(\Delta_{m,\mathbf a}^{(\mathbf r)})(\mathbf x)$.

Finally, we show that $5 \not\in T$.
For a contradiction, let $5=p_s$.
If $\widetilde{\lambda}(\Delta_{m,\mathbf a}^{(\mathbf r)})(\mathbf x)$ is $\z_2$-universal, then we may see that $5 \not\in T$ through same arguments with Lemma \ref{lem33}.
Now we assume that the generalized primitive regular shifted $m$-gonal form $\widetilde{\lambda}(\Delta_{m,\mathbf a}^{(\mathbf r)})(\mathbf x)=\Delta_{m,\mathbf a'}^{(\mathbf r')}(\mathbf x)$ is not $\z_2$-universal satisfying (\ref{lambda2 ex'qf}).

\begin{itemize}
\item[Case 1.] $\ord_5(a_2')=0$.\\
Recall that the generalized regular shifted $m$-gonal form $\widetilde{\lambda}(\Delta_{m,\mathbf a}^{(\mathbf r)})(\mathbf x)$ would (locally) represent some four integers having the same parity and having all different residues modulo $5$ in $\left[0,9\left(\le\frac{m-2}{2}\right)\right]$.
For the regularity of $\widetilde{\lambda}(\Delta_{m,\mathbf a}^{(\mathbf r)})(\mathbf x)=\Delta_{m,\mathbf a'}^{(\mathbf r')}(\mathbf x)$ with $a_3' \eq \cdots \eq a_n' \eq 0 \pmod{3}$, we have that the four non-negative integers
$$0, \ a_1'P_m^{(r_1')}(x_1), \ a_2'P_m^{(r_2')}(x_2), \ a_1'P_m^{(r_1')}(x_1)+a_2'P_m^{(r_2')}(x_2)$$
are less than $10$ and have all different residues modulo $5$ where $x_i=\sgn\left(\frac{m-2}{2}-r_i'\right)\cdot 1$.
If $a_3'>5$ or $r_3 \not\in \{1,m-3\}$, then from $(r_i',m-2)=1$, we obtain that $a_1'=2$ and $a_2'=4$ or $6$, which may yield that $\widetilde{\lambda}(\Delta_{m,\mathbf a}^{(\mathbf r)})(\mathbf x)$ is $\z_2$-universal.
So we have that $a_3'=5$ and $r_3' \in \{1,m-3\}$.
And we have that $a_1' \not\eq \pm a_2' \pmod{5}$ for $\widetilde{\lambda}(\Delta_{m,\mathbf a}^{(\mathbf r)})(\mathbf x)$ is not $\z_5$-universal.
In order for  the generalized shifted $m$-gonal form $\widetilde{\lambda}(\Delta_{m,\mathbf a}^{(\mathbf r)})(\mathbf x)$ which is not $\z_5$-universal to represent at least $4$ integers having the same parity and having all different residues modulo $5$ in $[0,9]$, $(a_1',a_2';r_1',r_2')$ should be $\left(1,2;\frac{m-2}{2}\pm\frac{m-4}{2},\frac{m-2}{2}\pm\frac{m-4}{2}\right)$.
The $\widetilde{\lambda}(\Delta_{m,\mathbf a}^{(\mathbf r)})(\mathbf x)$ with $(a_1',a_2',a_3';r_1',r_2',r_3') = \left(1,2,5;\frac{m-2}{2}\pm\frac{m-4}{2},\frac{m-2}{2}\pm\frac{m-4}{2}, \frac{m-2}{2}\pm\frac{m-4}{2} \right)$ locally represent $10$ or $15$, but the generalized shifted $m$-gonal form $\widetilde{\lambda}(\Delta_{m,\mathbf a}^{(\mathbf r)})(\mathbf x)$ does not globally represent both of $10$ and $15$.
Which is a contradiction to the regularity of $\widetilde{\lambda}(\Delta_{m,\mathbf a}^{(\mathbf r)})(\mathbf x)$.
\item[Case 2.] $\ord_5(a_2')>0$.\\
Since the generalized regular shifted $m$-gonal form $\widetilde{\lambda}(\Delta_{m,\mathbf a}^{(\mathbf r)})(\mathbf x)=\Delta_{m,\mathbf a'}^{(\mathbf r')}(\mathbf x)$ with $a_2' \eq \cdots \eq a_n' \eq 0 \pmod{5}$ locally represents at least two integers having the same parity and having all different residues modulo $5$ in $[0,9 \left( \le \frac{m-2}{2}\right)]$, we have that the non-negative integer
$$a_1'P_m^{(r_1')}(x_1)$$
is less than $10$ and a unit of $\z_5$
where $x_1=\sgn\left(\frac{m-2}{2}-r_1'\right)\cdot 1$.

We may have that $\ord_5(a_4') \not=1$ and $(\ord_5(a_3'), \ord_5(a_4')) \not= (1,2)$ for the regularity of $\widetilde{\lambda}(\Delta_{m,\mathbf a}^{(\mathbf r)})(\mathbf x)=\Delta_{m,\mathbf a'}^{(\mathbf r')}(\mathbf x)$.
Because otherwise, the number of different residues modulo $5$ of the integers in $[0,9 \left(\le \frac{m-2}{2}\right)]$ which are locally represented by $\widetilde{\lambda}(\Delta_{m,\mathbf a}^{(\mathbf r)})(\mathbf x)$ would be $3$.

If $\ord_5(a_3') =1$, then we may take three integers having the same parity and having all different residues modulo $5$ in $[0,19(<m)]$ which are locally represented by $\widetilde{\lambda}(\Delta_{m,\mathbf a}^{(\mathbf r)})(\mathbf x)$.
For the regularity of $\widetilde{\lambda}(\Delta_{m,\mathbf a}^{(\mathbf r)})(\mathbf x)=\Delta_{m,\mathbf a'}^{(\mathbf r')}(\mathbf x)$ with $a_2' \eq \cdots \eq a_n' \eq 0 \pmod{5}$, we have that
$$a_1'P_m^{(r_1')}(0)=0, \ a_1'P_m^{(r_1')}(x_1) <10, \text{ and } a_1'P_m^{(r_1')}(-x_1) <20$$ 
have all different residues modulo $5$  where $x_1=\sgn\left(\frac{m-2}{2}-r_1'\right)\cdot 1$.
So we may have that $a_1'=1$ and $m=28$ because 
$$a_1'P_m^{(r_1')}(x_1) +a_1'P_m^{(r_1')}(-x_1)=a_1'(m-2)$$ and 
$$15<a_1'P_m^{(r_1')}(-x_1) <20$$
for $x_1=\sgn\left(\frac{m-2}{2}-r_1'\right)\cdot 1$.
And since $(r_1',m-2)=1$, we may obtain that $\widetilde{\lambda}(\Delta_{m,\mathbf a}^{(\mathbf r)})(\mathbf x)=\Delta_{m,\mathbf a'}^{(\mathbf r')}(\mathbf x)$ represents every non-negative odd integer over $\z_2$.
On the other hand, since $\widetilde{\lambda}(\Delta_{m,\mathbf a}^{(\mathbf r)})(\mathbf x)$ represents every non-negative integer which is equivalent with 
$$a_1'P_m^{(r_1')}(0)=0 \ \text{ or } \ a_1'P_m^{(r_1')}\left(\sgn\left(\frac{m-2}{2}-r_1' \right)\cdot 1\right) <10$$
modulo $5$ over $\z_5$, $\widetilde{\lambda}(\Delta_{m,\mathbf a}^{(\mathbf r)})(\mathbf x)$ locally represents every non-negative odd multiple of $5$.
For the regularity of $\widetilde{\lambda}(\Delta_{m,\mathbf a}^{(\mathbf r)})(\mathbf x)=\Delta_{m,\mathbf a'}^{(\mathbf r')}(\mathbf x)$, we have that $a_2'=5$, $a_3'=10$  and $r_2',r_3' \in \{1,m-3\}$ for $\widetilde{\lambda}(\Delta_{m,\mathbf a}^{(\mathbf r)})(\mathbf x)$ locally represents $5$ and $15$.
And then we may obtain a contradiction to the regularity of $\widetilde{\lambda}(\Delta_{m,\mathbf a}^{(\mathbf r)})(\mathbf x)$ for $\widetilde{\lambda}(\Delta_{m,\mathbf a}^{(\mathbf r)})(\mathbf x)$ does not represent $25$ which is locally represented by $\widetilde{\lambda}(\Delta_{m,\mathbf a}^{(\mathbf r)})(\mathbf x)$.

Now we assume that $\ord_5(a_3')>1$.
Since $\widetilde{\lambda}(\Delta_{m,\mathbf a}^{(\mathbf r)})(\mathbf x)=\Delta_{m,\mathbf a'}^{(\mathbf r')}(\mathbf x)$ locally represents every odd multiple of $5$ or every even multiple of $5$, we have that $a_2'$ is $5$ or $10$.
And then $\widetilde{\lambda}(\Delta_{m,\mathbf a}^{(\mathbf r)})(\mathbf x)$ does not represent $15$ (resp, $20$) which is locally represented by $\widetilde{\lambda}(\Delta_{m,\mathbf a}^{(\mathbf r)})(\mathbf x)$ when $a_2'=5$ (resp, $10$), which is a contradiction to the regularity of $\widetilde{\lambda}(\Delta_{m,\mathbf a}^{(\mathbf r)})(\mathbf x)$.
\end{itemize}
\end{proof}

\vskip 1em

\section{Regular $m$-gonal forms}

In this section, we determine every type for generalized primitive regular $m$-gonal form of rank $n \ge 4$ for $m\ge 14$ with $m \not\eq 0 \pmod{4}$ and $m \ge 28$ with $m \eq 0 \pmod{4}$ by showing Theorem \ref{main}.
We split the proof of Theorem \ref{main} into four Theorem \ref{locuni}, Theorem \ref{notuni3}, Theorem \ref{notuni2}, and Theorem \ref{thm m41} depending on the residue conditions of $m$ modulo $4$ and $3$.

\begin{thm}\label{locuni}
For $m\ge 14$ with $m\not\equiv0 \pmod{4}$ and $m-2\equiv0 \pmod{3}$, a generalized primitive regular shifted $m$-gonal form of rank $n \ge 4$ is universal. 
\end{thm}

\begin{proof}
By Proposition \ref{prop1} and Lemma \ref{lem33}, a generalized primitive regular shifted $m$-gonal form is locally universal for $m\ge 14$ with $m\not\equiv0 \pmod{4}$ and $m-2\equiv0 \pmod{3}$.
This completes the claim.
\end{proof}

\begin{thm} \label{notuni3}
For $m\ge 14$ with $m\not\equiv0 \pmod{4}$ and $m-2\not\equiv0 \pmod{3}$, a generalized primitive regular $m$-gonal form $\Delta_{m,\mathbf a}(\mathbf x)$ is universal or of the form of (\ref{1.notuni/3}).
\end{thm}
\begin{proof}
By Lemma \ref{lem33}, for $m \ge 14$ with $m\not\equiv0 \pmod{4}$ and $m-2\not\equiv0 \pmod{3}$, a generalized primitive regular $m$-gonal form $\Delta_{m,\mathbf a}(\mathbf x)$ is universal over $\z_p$ for every prime $p \not= 3$.
If $\Delta_{m,\mathbf a}(\mathbf x)$ is universal over $\z_3$ too, then the generalized primitive regular $m$-gonal form $\Delta_{m,\mathbf a}(\mathbf x)$ would be (locally) universal.
From now on, we consider a generalized primitive regular $m$-gonal form $$\Delta_{m,\mathbf a}(\mathbf x)=a_1P_m(x_1)+\cdots+a_nP_m(x_n)$$ which is not $\z_3$-universal.
By Remark \ref{rmk1}, there would be at most two units of $\z_3$ in $\{a_1,\cdots,a_n\}$ by admitting a recursion because ternary unimodular quadratic forms over $\z_3$ are universal.
We split the proof into two parts depending on the residue of $m$ modulo $3$.

First we treat the case of $m\equiv 1 \pmod{3}$.
Over $\z_3$, the (generalized) $m$-gonal number with $m\equiv 1 \pmod{3}$ $$P_m(x)=\frac{m-2}{2}\left(x-\frac{m-4}{2(m-2)}\right)^2-\frac{(m-4)^2}{8(m-2)}$$ represents every $3$-adic integer in $1+3\z_3$ because $\frac{m-2}{2}\left(\z_3^{\times}-\frac{m-4}{2(m-2)}\right)^2=1+3\z_3$ and $\frac{(m-4)^2}{8(m-2)} \in 3\z_3$.

{\bf First, assume that there is only one unit of $\z_3$ in $\{a_1,a_2,\cdots, a_n\}$ by admitting a recursion.}\\
Without loss of generality, let $a_1 \in \z_3^{\times}$ and $3\le a_2\le \cdots \le a_n$ be multiples of $3$.
Then $$\Delta_{m,\mathbf a}(\mathbf x)=a_1P_m(x_1)+\cdots+a_nP_m(x_n)$$ represents every $3$-adic integer in $a_1+3\z_3$ over $\z_3$ because $P_m(x_1)$ represents every $3$-adic integer in $1+3\z_3$ over $\z_3$.
So the generalized primitive regular $m$-gonal form $\Delta_{m,\mathbf a}(\mathbf x)$ (which is $\z_p$-universal for every $p \not=3$) would (locally) represent every non-negative integer $n \in \N_0$ with $n \eq a_1 \pmod{3}$.
Hence $a_1$ should be less than $3$ because otherwise, $\Delta_{m,\mathbf a}(\mathbf x)$ does not represent the smallest non-negative integer which is equivalent with $a_1$ modulo $3$.
For the regular $\Delta_{m,\mathbf a}(\mathbf x)$ represents $a_1+3, \ a_1+6,$ and $a_1+9\left(<\frac{m-3}{a_1}\right)$,
$(a_2,a_3,a_4)$ would be $(3,3,3),(3,3,6),(3,3,9)$ or $(3,6,a_4)$.
Since the $3$-adic integers which are represented by the diagonal quadatic form
$$\Delta_{4,\mathbf a}\otimes\z_3(=\left<a_1,\cdots ,a_n\right>\otimes \z_3)$$ 
with $a_1 \in \z_3^{\times}$, $(a_2,a_3,a_4) \in \{(3,3,3),(3,3,6),(3,3,9),(3,6,a_4)\}$ and $a_2\eq \cdots \eq a_n \eq 0 \pmod{3}$ are $a_1+3\z_3 \cup 3\z_3$, from (\ref{m4}), we have that the $3$-adic integers which are represented by the $m$-gonal form $\Delta_{m,\mathbf a}(\mathbf x)$ over $\z_3$ are $a_1+3\z_3 \cup 3\z_3$.
Which implies that non-negative integers which are locally represented by $\Delta_{m,\mathbf a}(\mathbf x)$ are
$$\{n \in \N_0|n \not\eq 2a_1 \pmod{3}\}$$
for $\Delta_{m,\mathbf a}(\mathbf x)$ is $\z_p$-universal for every $p\not=3$.
Consequently, we may conclude that a generalized primitive non-universal regular $m$-gonal form is of the form of (\ref{1.notuni/3}) in this case.

{\bf Secondly, assume that there are two units of $\z_3$ in $\{a_1,a_2,\cdots, a_n\}$ by admitting a recursion.}\\
Without loss of generality, let $a_1\le a_2\in \z_3^{\times}$ and $3 \le a_3\le \cdots \le a_n$ be multiples of $3$.
From the assumption that $\Delta_{m,\mathbf a}(\mathbf x)$ is not $\z_3$-universal (and so does the diagonal quadratic form $\Delta_{4,\mathbf a}(\mathbf x)$ by Remark \ref{rmk1}), we have that $a_1\eq a_2 \pmod{3}$.
Since the diagonal quadratic form
$$\Delta_{4,\mathbf a}\otimes \z_3(=\left<a_1,\cdots,a_n\right>\otimes \z_3 \supset \left<a_1,a_2\right>\otimes \z_3)$$
represents every unit of $\z_3$, the generalized $m$-gonal form $\Delta_{m,\mathbf a}(\mathbf x)$ also represents every unit of $\z_3$ over $\z_3$ by (\ref{m4}) with $\frac{(m-4)^2}{8(m-2)}(a_1+\cdots+a_n)\in 3\z_3$.
So the regular $\Delta_{m,\mathbf a}(\mathbf x)$ would (locally) represent every non-negative integer which is relatively prime with $3$.
For $\Delta_{m,\mathbf a}(\mathbf x)$ represents $1$ and $2$, $(a_1,a_2)$ should be $(1,1)$ for $a_1\eq a_2 \pmod{3}$.
For $\Delta_{m,\mathbf a}(\mathbf x)$ represents $4$ and $7$, $(a_3,a_4)$ should be $(3,3)$ or $(3,6)$.
And then from Remark \ref{rmk1}, we arrive at a contradiction to our assumption that $$\Delta_{m,\mathbf a}(\mathbf x)=a_1P_m(x_1)+\cdots+a_nP_m(x_n)$$
is not $\z_3$-universal because having a binary unimodular subform $\left<a_1,a_2\right>\otimes \z_3$ and a binary $3$-modular subform $\left<a_3,a_4\right>\otimes \z_3$ diagonal quadratic form 
$$\Delta_{4,\mathbf a}\otimes\z_3(=\left<a_1,\cdots,a_n\right>\otimes \z_3)$$
is universal.
Thus in this case, there is no generalized primitive regular $m$-gonal form.

\vskip 0.8em

Next suppose that $m\equiv 0 \pmod{3}$.
Over $\z_3$, the (generalized) $m$-gonal number with $m\equiv 0 \pmod{3}$ 
$$P_m(x)=\frac{m-2}{2}\left(x-\frac{m-4}{2(m-2)}\right)^2-\frac{(m-4)^2}{8(m-2)}$$
represents every integer in $3\z_3$ because $\frac{m-2}{2}\left(\z_3^{\times}-\frac{m-4}{2(m-2)}\right)^2=-1+3\z_3$ with $\frac{(m-4)^2}{8(m-2)} \in 1+3\z_3$.

{\bf First assume that there is only one unit of $\z_3$ in $\{a_1,a_2,\cdots, a_n\}$ by admitting a recursion.}\\
Without loss of generality, let $a_1 \in \z_3^{\times}$ and $3 \le a_2\le \cdots \le a_n$ be multiples of $3$.
Since $P_m(x)$ represents every $3$-adic integer in $3\z_3$ over $\z_3$, the generalized primitive regular $m$-gonal form $\Delta_{m,\mathbf a}(\mathbf x)$ (which is $\z_p$-universal for every $p \not=3$) would (locally) represent every non-negative multiple of $3$.
For the regular $\Delta_{m,\mathbf a}(\mathbf x)$ represents $3,\ 6,$ and $9(<m-3)$, $(a_2,a_3,a_4)$ would be $(3,3,3),\ (3,3,6), \ (3,3,9)$ or $(3,6,a_4)$.
Since the $3$-adic integers which are represented by the diagonal quadatic form
$$\Delta_{4,\mathbf a}\otimes\z_3(=\left<a_1,\cdots ,a_n\right>\otimes \z_3)$$ 
with $a_1 \in \z_3^{\times}$, $(a_2,a_3,a_4) \in \{(3,3,3),(3,3,6),(3,3,9),(3,6,a_4)\}$ and $a_2\eq \cdots \eq a_n \eq 0 \pmod{3}$ are $a_1+3\z_3 \cup 3\z_3$, from (\ref{m4}), we have that the $3$-adic integers which are represented by the generalized $m$-gonal form $\Delta_{m,\mathbf a}(\mathbf x)$ over $\z_3$ are $a_1+3\z_3 \cup 3\z_3$.
Which implies that the non-negative integers which are locally represented by the (generalized) $m$-gonal form $\Delta_{m,\mathbf a}(\mathbf x)$ are
$$\{n \in \N_0|n \not\eq 2a_1 \pmod{3}\}$$
for $\Delta_{m,\mathbf a}(\mathbf x)$ is $\z_p$-universal for every $p\not=3$.
Consequently, we may conclude that a generalized primitive non-universal regular $m$-gonal form is of the form of (\ref{1.notuni/3}) in this case.

{\bf Secondly, assume that there are two units of $\z_3$ in $\{a_1,a_2,\cdots, a_n\}$ by admitting a recursion.}\\
Without loss of generality, let $a_1\le a_2\in \z_3^{\times}$ and $3\le a_3\le \cdots \le a_n$ be multiples of $3$.
From the assumption that $\Delta_{m,\mathbf a}(\mathbf x)$ is not $\z_3$-universal, we have that $a_1\eq a_2 \pmod{3}$.
Since $\Delta_{m,\mathbf a}(\mathbf x)$ locally represents $3$ and $6$, we have that $(a_3,a_4)=(3,3)$ or $(3,6)$ for its regularity.
And then we obtain a contradiction to our assumption that $\Delta_{m,\mathbf a}(\mathbf x)$ is not $\z_3$-universal with universal $\Delta_{4,\mathbf a}\otimes \z_3$ from Remark \ref{rmk1}.
Thus in this case, there is  no generalized primitive regular $m$-gonal form which is not universal.
\end{proof}

\vskip 0.8em

\begin{rmk} \label{rank rmk1}
By using Theorem \ref{uni min}, we may see that there is a generalized regular $m$-gonal form of rank $\ceil{\log_2(m+2)}+1$ of the form of (\ref{1.notuni/3}) with universal generalized $m$-gonal form $\left<a_2',\cdots,a_n'\right>_m$ of rank $n-1=\ceil{log_2(m+2)}$.
On the other hand, in order for a generalized regular $m$-gonal form $\Delta_{m.\mathbf a}(\mathbf x)$ of the form of (\ref{1.notuni/3}) to represent every positive integer in $3\z$ (resp, $a_1+3\z$) when $m \eq 1 \pmod{3}$ (resp, $m \eq 0 \pmod{3}$), the rank of $\Delta_{m.\mathbf a}(\mathbf x)$ must be at least $\ceil{\log_2(m-3)}+1$.
Because the generalized $m$-gonal form $\left<3a_2',\cdots,3a_n'\right>_m$ should represent at least every non-negative integer in $3\z$ up to $3(m-2)$ (resp, $3(m-4)$) in order for the generalized regular $m$-gonal form $\Delta_{m.\mathbf a}(\mathbf x)$ of the form of (\ref{1.notuni/3}) to represent every positive integer in $3\z$ (resp, $a_1+3\z$) when $m \eq 1 \pmod{3}$ (resp, $m \eq 0 \pmod{3}$).
Therefore by Remark \ref{rank rmk1}, the minimal rank of a generalized regular $m$-gonal form would be $\ceil{\log_2(m+2)}$ and moreover, for most of $m$, every generalized regular $m$-gonal form $\Delta_{m.\mathbf a}(\mathbf x)$ of minimal rank would be universal form.
Determining the minimal rank between $\ceil{\log_2(m-3)}+1$ and $\ceil{\log_2(m+2)}+1$ of a generalized regular $m$-gonal form of the form of (\ref{1.notuni/3}) would be one of interesing problems.
\end{rmk}

\begin{rmk}
For a generalized regular $m$-gonal form $\Delta_{m,\mathbf a}(\mathbf x)$ of the form of (\ref{1.notuni/3}),
its $\lambda_3$-transformation $\lambda_3(\Delta_{m,\mathbf a})(\mathbf x)$ is regular and of the form of
$$\left<3a_1^{(r_1)},a_2'^{(1)},a_3'^{(1)},\cdots, a_n'^{(1)} \right>_m$$
where $r_1=\frac{2m-2}{6} \in \N$ ($r_1=\frac{4m-6}{6} \in \N$) when $m\eq 1 \pmod{3}$ (resp, $m \eq 0 \pmod{3}$).
From the fact that $\left<a_1,3a_2,\cdots,3a_n\right>_m$ represents every $3$-adic integer in $3\z_3$ (resp, $a_1+3\z_3$) over $\z_3$ when $m\eq 1 \pmod{3}$ (resp, $m \eq 0 \pmod{3}$), we may yield that the regular $\lambda_3(\Delta_{m,\mathbf a})(\mathbf x)$ is $\z_3$-universal.
Moreover, since the generalized (shifted) $m$-gonal form $\Delta_{m,\mathbf a}(\mathbf x)$ is locally universal.
Consequently, we may conclude that the regular $\lambda_3(\Delta_{m,\mathbf a})(\mathbf x)$ is universal.

On the other hand, one may naturally wonder whether converse also holds, i.e., whether if a generalized shifted $m$-gonal form
$$\left<3a_1^{(r_1)},a_2'^{(1)},a_3'^{(1)},\cdots, a_n'^{(1)} \right>_m$$ where $a_1 \in \{1,2\}$ and $r_1=\frac{2m-2}{6} \in \N$ ($r_1=\frac{4m-6}{6} \in \N$) when $m\eq 1 \pmod{3}$ (resp, $m \eq 0 \pmod{3}$) is universal, then a generalized $m$-gonal form $$\left<a_1^{(1)},3a_2'^{(1)},3a_3'^{(1)},\cdots, 3a_n'^{(1)} \right>_m$$ is regular.
The non-negative integers which are locally represented by generalized $m$-gonal form $\left<a_1^{(1)},3a_2'^{(1)},3a_3'^{(1)},\cdots, 3a_n'^{(1)} \right>_m$ with generalized universal shifted $m$-gonal form $\left<3a_1^{(r_1)},a_2'^{(1)},a_3'^{(1)},\cdots, a_n'^{(1)} \right>_m$ are $$a_1+3\N_0 \cup 3\N_0.$$
Since $\left<9a_1^{(r_1)},3a_2'^{(1)},3a_3'^{(1)},\cdots, 3a_n'^{(1)} \right>_m$ represents every multiple of $3$, we may see that $\left<a_1^{(1)},3a_2'^{(1)},3a_3'^{(1)},\cdots, 3a_n'^{(1)} \right>_m$ represents every non-negative integer in $3\N_0$ (resp, $a_1+3\N_0$) for $P_m(3x_1)=9P_m^{(r_1)}(x_1)$ (resp, $P_m(3x_1+1)=9P_m^{(r_1)}(x_1)+P_m(1)=9P_m^{(r_1)}(x_1)+1$) when $m \eq 1 \pmod{3}$ (resp, $m \eq 0 \pmod{3}$).
Meanwhile, the authors didn't get a specific ground for that the generalized $m$-gonal form  $\left<a_1^{(1)},3a_2'^{(1)},3a_3'^{(1)},\cdots, 3a_n'^{(1)} \right>_m$ with universal $\left<3a_1^{(r_1)},a_2'^{(1)},a_3'^{(1)},\cdots, a_n'^{(1)} \right>_m$ represents every non-negative integer in $a_1+3\N_0$ (resp, $3\N_0$) when $m \eq 1 \pmod{3}$ (resp, $m \eq 0 \pmod{3}$), even though it is likely.
\end{rmk}

\vskip 1.5em

\begin{thm}\label{notuni2}
For $m \ge 28$ with $m \equiv0 \pmod{4}$ and $m \equiv 2 \pmod{3}$, a generalized primitive regular $m$-gonal form $\Delta_{m,\mathbf a}(\mathbf x)$ of rank $n\ge4 $ is universal or of the form of (\ref{1.notuni/4m8}), (\ref{3.notuni/4m8}), or (\ref{4.notuni/4m8}).
\end{thm}
\begin{proof}
We prove the theorem only for $m\ge 28$ with $m \eq 4 \pmod{8}$ and $m \eq 2\pmod{3}$ in this paper.
One may prove the theorem of the other case (i.e., for $m \ge 28$ with $m \eq 0 \pmod{8}$ and $m \eq 2 \pmod{3}$) through similar processing with this proof.

Let $\Delta_{m,\mathbf a}(\mathbf x)$ be a generalized primitive regular $m$-gonal form  for $m \ge 28$ with $m \equiv4 \pmod{8}$ and $m \equiv 2 \pmod{3}$.
By Proposition \ref{prop1} and Lemma \ref{lem33}, the generalized primitive regular $m$-gonal form $\Delta_{m,\mathbf a}(\mathbf x)$ is $\z_p$-universal for every odd prime $p$.
If $\Delta_{m,\mathbf a}(\mathbf x)$ is universal over $\z_2$ too, then the generalized regular $m$-gonal form $\Delta_{m,\mathbf a}(\mathbf x)$ would be (locally) universal.
From now on, we consider a generalized primitive regular $m$-gonal form $$\Delta_{m,\mathbf a}(\mathbf x)=a_1P_m(x_1)+\cdots+a_nP_m(x_n)$$ which is not $\z_2$-universal.
By Remark \ref{rmk1}, since the $\z_2$-universality of $\Delta_{m,\mathbf a}(\mathbf x)$ and the $\z_2$-universaility of the diagonal quadratic form $\Delta_{4,\mathbf a}(\mathbf x)$ are equivalent, we have that there are at most three units of $\z_2$ in $\{a_1,\cdots, a_n\}$ by admitting a recursion and for the primitivity of $\Delta_{m,\mathbf a}(\mathbf x)$, there would be at least one unit of $\z_2$ in $\{a_1,\cdots, a_n\}$.

{\bf First assume that there is only one unit of $\z_2$ in $\{a_1,\cdots, a_n\}$ by admitting a recursion.}\\
Without loss of generality, we assume that $a_1 \in \z_2^{\times}$ and $2 \le a_2\le \cdots \le a_n$ are even.
From $\left(\z_2^{\times}-\frac{m-4}{2(m-2)}\right)^2=1+8\z_2$, we may see that $$a_1P_m(x_1)=a_1\left\{ \frac{m-2}{2}\left(x_1-\frac{m-4}{2(m-2)}\right)^2-\frac{(m-4)^2}{8(m-2)} \right\}$$ represents every $2$-adic integer in $a_1+8\z_2$ over $\z_2$, so does $\Delta_{m,\mathbf a}(\mathbf x)$.
Therefore the generalized regular $m$-gonal form $\Delta_{m,\mathbf a}(\mathbf x)$ (which is $\z_p$-universal for every odd prime $p$) would (locally) represent every non-negative integer $n$ with $n \eq a_1 \pmod{8}$.
So $a_1$ should be less than $8$ because otherwise, $\Delta_{m,\mathbf a}(\mathbf x)$ could not represent the smallest non-negative integer which is equivalent with $a_1$ modulo $8$.
For the generalized regular $m$-gonal form $\Delta_{m,\mathbf a}(\mathbf x)$ represents every non-negative integer which is equivalent with $a_1$ modulo $8$ less than $a_1(m-3)$, the generalized $m$-gonal form
$$a_2P_m(x_2)+\cdots+a_nP_m(x_n)$$ would represent every positive multiple of $8$ less than $a_1(m-3)-a_1=a_1(m-4) \le 24$.
So $a_2$ could not exceed $8$, i.e., $a_2 \in \{2,4,6,8\}$.
When $a_2<8$, $a_3$ could not exceed $8$ again, i.e., $a_3 \in \{a_2,a_2+2,\cdots,8\}$
$\cdots$.
Through such the Bhargava's escalating processings, we may obtain that when $24<a_1 (m-3)-a_1$ (namely, when $(m;a_1)\not=(28;1)$), all the possible candidates for $(a_2,a_3,a_4)$ are followings
$$
\begin{array} {lllllll}
&(2,2,2),     &(2,2,4),    &(2,2,6),     &(2,2,8),     &(2,4,4),      &(2,4,6),      \\
&(2,4,8),     &(2,6,6),    &(2,6,8),    &(2,6,10),   &(2,6,12),    & (2,6,14),   \\
&(2,6,16),    &(2,8,8),   &(2,8,10),   &(2,8,12),   &(2,8,14),    &(2,8,16),     \\
&(4,4,4),     &(4,4,6),   &(4,4,8),     &(4,4,10),   &(4,4,12),   &(4,4,14),     \\
&(4,4,16),   &(4,6,6),    &(4,6,8),     &(4,8,8),    &(4,8,10),   &(4,8,12),    \\
&(4,8,14),   &(4,8,16),  &(6,6,6),     &(6,6,8),    &(6,8,8),     &(6,8,10),    \\
&(6,8,12),   &(6,8,14),  &(6,8,16),   &(8,8,8),    &(8,8,10),    &(8,8,12),   \\
&(8,8,14),   &(8,8,16),   &(8,8,18),   &(8,8,20),   &(8,8,22),    &(8,8,24),  \\
&(8,10,10), &(8,10,12), &(8,10,14),  &(8,10,16), &(8,12,12), &(8,12,14),   \\
&(8,12,16), &(8,14,14), &(8,14,16), &(8,16,16),  &(8,16,18), &(8,16,20),  \\
&(8,16,22), &(8,16,24). &              &                &              &
\end{array}
$$
For each candidate for $(a_2,a_3,a_4)$ with $a_2 \ge 6$, we could verify that $\Delta_{m,\mathbf a}(\mathbf x)$ represents every $2$-adic integer in  $4\z_2$ over $\z_2$ from (\ref{m4}) because the diagonal quadratic form
$$\Delta_{4,\mathbf a}\otimes \z_2=(\left<a_1,\cdots,a_n\right>\otimes \z_2\supseteq\left<a_1,a_2,a_3,a_4\right>\otimes \z_2)$$ represents every $2$-adic integer in $4\z_2$ over $\z_2$.
Therefore the generalized $m$-gonal form $\Delta_{m,\mathbf a}(\mathbf x)$ with $a_2 \ge 6$ locally represents every non-negative multiple of $4$.
But the generalized $m$-gonal form $$\Delta_{m,\mathbf a}(\mathbf x)=a_1P_m(x_1)+\cdots+a_nP_m(x_n)$$ with $a_1 \in \z_2^{\times}$ and $6 \le a_2 \le \cdots \le a_n$ could not represent $4$, so we obtain that $a_2<6 $ for the regularity of $\Delta_{m,\mathbf a}(\mathbf x)$.\\
Similarly with the case of $a_2 \ge 6$, for each candidate for $(a_2,a_3,a_4)$ with $a_2=2$, we could verify that the regular $\Delta_{m,\mathbf a}(\mathbf x)$ (locally) represents every non-negative even integer from (\ref{m4}).
So when $a_2=2$, $a_3$ could not exceed $4$, i.e., $(a_2,a_3)=(2,2) \text{ or } (2,4).$
For each candidate for $(a_2,a_3,a_4)$ with $(a_2,a_3) \in \{(2,2),(2,4)\}$ except $(2,2,8)$,  we may easily verify that the diagonal quadratic form $\Delta_{4,\mathbf a}\otimes \z_2(=\left<a_1,\cdots,a_n\right>\otimes \z_2  \supseteq \left<a_1,a_2,a_3,a_4\right>\otimes \z_2 )$ is universal, so does $\Delta_{m,\mathbf a}(\mathbf x)$ over $\z_2$ by Remark \ref{rmk1}, which is a contradiction to our assumption.
When $(a_2,a_3,a_4)=(2,2,8)$, $\Delta_{m,\mathbf a}(\mathbf x)$ does not represent $6$ which is locally represented by $\Delta_{m,\mathbf a}(\mathbf x)$, which is a contradiction to the regularity of $\Delta_{m,\mathbf a}(\mathbf x)$.
Therefore we finally have that $a_2 = 4$.\\
For each candidate for $(a_3,a_4)$ with $a_3\eq 2$ or $a_4\eq 2 \pmod{4}$, we could verify that the diagonal quadratic form
$\left<a_1,a_2,a_3,a_4\right>\otimes \z_2$ is universal, so does $\Delta_{m,\mathbf a}(\mathbf x)$ over $\z_2$, which is a contradiction to our assumption.
For each candidate $(a_2,a_3,a_4)$ with $a_2=4$ and $a_3 \eq a_4\eq 0 \pmod{4}$, if there is $a_i \in \{a_5,\cdots,a_n\}$ with $a_i \eq 2\pmod{4}$, then the diagonal quadratic form $\Delta_{4,\mathbf a}(\mathbf x)$ would be universal over $\z_2$, so does $\Delta_{m,\mathbf a}(\mathbf x)$.
Eventually, we have that $a_i \eq 0 \pmod {4}$ for all $2 \le i \le n$ for our assumption and the non-negative integers which are locally represented by $\Delta_{m,\mathbf a}(\mathbf x)$ are 
$$\{n \in \N_0|n\eq a_1 \text{ or }0 \pmod{4}\}$$
from (\ref{m4}) because the $2$-adic integers which are represented by the diagonal quadratic form $$\Delta_{4,\mathbf a}\otimes\z_2(=\left<a_1,\cdots ,a_n\right>\otimes \z_2)$$ are $a_1+4\z_2 \cup 4\z_2$.
Consequently, we may conclude that a generalized primitive regular $m$-gonal form is of the form of (\ref{1.notuni/4m8}) in this case.

It remains to resolve the claim for case of $24<a_1(m-3)-a_1$, i.e., $(m;a_1)=(28;1)$.
Since the generalized $m$-gonal form $\Delta_{m,\mathbf a}(\mathbf x)$ with $(m;a_1)=(28;1)$ locally represents every non-negative integer $n$ with $n \eq a_1=1 \pmod{8}$, for the regularity of the $\Delta_{m,\mathbf a}(\mathbf x)$, the generalized $m$-gonal form $a_2P_m(x_2)+\cdots+a_nP_m(x_n)$ would represent $8$ and $16(<a_1(m-3)-a_1)$.
Among candidates
$$
\begin{array} {lllllll}
&(2,2,2),     &(2,2,4),    &(2,2,6),     &(2,2,8),     &(2,4,4),      &(2,4,6),      \\
&(2,4,8),     &(2,6,6),    &(2,6,8),    &(2,6,10),   &(2,6,12),    & (2,6,14),   \\
&(2,6,16),    &(2,8,8),   &(2,8,10),   &(2,8,12),   &(2,8,14),    &(2,8,16),     \\
&(4,4,4),     &(4,4,6),   &(4,4,8),     &(4,4,10),   &(4,4,12),   &(4,4,14),     \\
&(4,4,16),   &(4,6,6),    &(4,6,8),     &(4,8,8),    &(4,8,10),   &(4,8,12),    \\
&(4,8,14),   &(4,8,16),  &(6,6,6),     &(6,6,8),    &(6,8,8),     &(6,8,10),    \\
&(6,8,12),   &(6,8,14),  &(6,8,16),   &(8,8,a_4),    &(8,16,a_4)    &   \\
\end{array}
$$
for $(a_2,a_3,a_4)$ in order for the generalized $m$-gonal form $a_2P_m(x_2)+\cdots+a_nP_m(x_n)$ to represent $8$ and $16$, first consider
$$(8,8,a_4) \text{ and } (8,16,a_4).$$

If $(a_2,a_3)=(8,16)$, then we may use (\ref{m4}) to see that $\Delta_{m,\mathbf a}(\mathbf x)$ represents every non-negative multiple of $8$ over $\z_2$ because the diagonal quadratic form $$\Delta_{4,\mathbf a}\otimes\z_2(=\left<a_1,\cdots ,a_n\right>\otimes \z_2\supset \left<a_1,a_2,a_3\right>\otimes \z_2)$$ 
represents every $2$-adic integer in $4\z_2^{\times}$ and $\frac{(m-4)^2}{8(m-2)}(a_1+\cdots+a_n) \in 4\z_2^{\times}$.
For the regularity of $\Delta_{m,\mathbf a}(\mathbf x)$, $a_4$ could not exceed $32 \in 8\z$.
For each $a_4 \in \{16,18,\cdots, 32\}$ with $(a_1,a_2,a_3)=(1,8,16)$, we could verify that $$\Delta_{4,(a_1,a_2,a_3,a_4)}\otimes \z_2 = \left<a_1,a_2,a_3,a_4\right>\otimes \z_2$$
represents every $2$-adic integer in $4\z_2$, so does $\Delta_{4,\mathbf a}\otimes \z_2$.
Which induces that $\Delta_{m,\mathbf a}(\mathbf x)$ locally represents every non-negative multiple of $4$ from (\ref{m4}), yielding a contradiction to the regularity of $\Delta_{m,\mathbf a}(\mathbf x)$ because $\Delta_{m,\mathbf a}(\mathbf x)$ with $a_1=1$ and $8 \le a_2 \le \cdots \le a_n$ does not represent $4$.
 

If $(a_2,a_3)=(8,8)$, then since the generalized $m$-gonal form $a_1P_m(x_1)+a_2P_m(x_2)+a_3P_m(x_3)$ does not represent $49( \eq a_1 \pmod{8})$, for the regularity of $\Delta_{m,\mathbf a}(\mathbf x)$, $a_4$ could not exceed $49$.
For each $8 \le a_4  \le 48$ with $a_4 \eq 0 \pmod{2}$, since the diagonal quadraitc form
$$\Delta_{4,(a_1,a_2,a_3,a_4)}\otimes \z_2 = \left<a_1,a_2,a_3,a_4\right>\otimes \z_2$$
represents every $2$-adic integer in $8\z_2$, the generalized $m$-gonal form $\Delta_{m,\mathbf a}(\mathbf x)$ would locally represent every non-negative integer in $4+8\z$, which is a contradiction to the regularity of $\Delta_{m,\mathbf a}(\mathbf x)$.
For the remaining candidates $(a_2,a_3,a_4)$ with $(a_2,a_3,a_4) \not= (8,8,a_4),(8,16,a_4)$, through same processing with the case of $24<a_1(m-3)-a_1$, we may conclude that a generalized primitive regular $m$-gonal form is of the form of (\ref{1.notuni/4m8}).

{\bf Secondly, assume that there are two units of $\z_2$ in $\{a_1,\cdots, a_n\}$ by admitting a recursion.} \\
Without loss of generality, let $a_1 \le a_2 \in \z_2^{\times}$ and $2 \le a_3 \le \cdots \le a_n$ be even.
\begin{itemize}
\item[Case 1.] $\left<a_1,\cdots,a_n\right>\otimes \z_2 \supset \left<u,u'\right>\otimes \z_2$ for some $u \equiv 1 \text{ and } u'\equiv3 \pmod{4}$.\\
In this case, the generalized $m$-gonal form $\Delta_{m,\mathbf a}(\mathbf x)$ locally represents every odd non-negative integer by
(\ref{m4}) because the diagonal quadratic form
$$\left<a_1,\cdots, a_n\right>\otimes \z_2 \supset \left<u,u'\right>\otimes \z_2$$ represents every unit of $\z_2$.
For the regularity of $\Delta_{m,\mathbf a}(\mathbf x)$, we have that $a_1=1$.
Since the generalized regular $m$-gonal form $\Delta_{m,\mathbf a}(\mathbf x)$ (locally) represents $3,\ 5,\ 7,$ and $9$, $(a_2,a_3,a_4)$ would be one of the followings
$$
\begin{array} {llllllll}
&(a_2,2,2),     &(a_2,2,4),    &(3,2,6),     &(3,4,4),     &(3,4,6),      &(3,4,8), &(5,2,8)    \\
\end{array}
$$
where $a_2$ is an odd non-negative integer.
For each candidate for $(a_2,a_3,a_4)$ except $(3,4,4)$ and $(3,4,8)$,
we could verify that the diagonal quadratic forms
$$\Delta_{4,\mathbf a}\otimes \z_2=\left<a_1,\cdots,a_n\right>\otimes \z_2 \supseteq \left<a_1,a_2,a_3,a_4\right>\otimes \z_2$$ are universal, which yields that $\Delta_{m,\mathbf a}(\mathbf x)$ is $\z_2$-universal by Remark \ref{rmk1}.
So we have $(a_2,a_3,a_4)=(3,4,4)$ or $(3,4,8)$ and $a_i\equiv 0 \pmod{4}$ for all $3 \le i \le n$ for our assumption that $\Delta_{m,\mathbf a}(\mathbf x)$ is not $\z_2$-universal.
For each $\mathbf a=(a_1,\cdots,a_n)$ with $(a_1,a_2,a_3,a_4) \in \{(1,3,4,4),(1,3,4,8)\}$ and $a_i \eq 0 \pmod{4}$ for all $3\le i \le n$, the non-negative integers which are locally represented by $\Delta_{m,\mathbf a}(\mathbf x)$ are 
$$\{n \in \N_0| n \not\eq 2\pmod{4}\}$$
because the $2$-adic integers which are represented by $\Delta_{4,\mathbf a}(\mathbf x)$ are $\z_2 \setminus 2\z_2^{\times}$ over $\z_2$ and $\frac{(m-4)^2}{8(m-2)}(a_1+\cdots+a_n) \in 4\z_2$.
So we may conclude that a generalized primitive regular $m$-gonal form is of the form of (\ref{3.notuni/4m8}) in this case.

\item[Case 2.] $\left<a_1,\cdots,a_n\right>\otimes \z_2 \not\supset \left<u,u'\right>\otimes \z_2$ for any $u \equiv 1 \text{ and } u'\equiv3 \pmod{4}$.\\
For our assumption, we have that 
$$a_1\eq a_2 \pmod{4} \text{ and } a_3\eq\ \cdots\eq a_n \eq 0 \pmod{4}.$$
Since the diagonal quadraitc forms $$\Delta_{4,\mathbf a}\otimes \z_2=\left<a_1,\cdots,a_n\right>\otimes \z_2 \supset \left<a_1,a_2\right>\otimes \z_2$$
represent every $2$-aidc integer in $a_1+4\z_2$, we have that $\Delta_{m,\mathbf a}(\mathbf x)$ locally represents every non-negative integer which is equivalent with $a_1$ modulo $4$ from (\ref{m4}).
For the regularity of $\Delta_{m,\mathbf a}(\mathbf x)$, $a_1$ would be less than $4$, i.e, $a_1=1$ or $3$.

If $a_1=1$, then since the generalized regular $m$-gonal form $\Delta_{m,\mathbf a}(\mathbf x)$ (locally) represents $1,5,9,13, \cdots$, $(a_2,a_3,a_4)$ would be one of the followings
$$
\begin{array} {lllllll}
&(a_2,4,4),&(a_2,4,8),&(5,8,8),&(5,8,12),&(5,8,16), &(9,4,a_4)\\
\end{array}
$$
where $a_2 \eq 1 \pmod{4}$ and $a_4 \eq 0\pmod{4}$. 
For each candidate for $(a_2,a_3,a_4)$, we could verify that the diagonal quadraitc forms 
$$\Delta_{4,\mathbf a}\otimes \z_2=\left<a_1,\cdots,a_n\right>\otimes \z_2 \supseteq \left<a_1,\cdots,a_4\right>\otimes \z_2$$
represent every $2$-adic integer in $4\z_2$, which yields that the generalized $m$-gonal form $\Delta_{m,\mathbf a}(\mathbf x)$ locally represents every non-negative multiple of $4$.
Therefore for the regularity of $\Delta_{m,\mathbf a}(\mathbf x)$, we have that $(a_3,a_4)=(4,4)$ or $(4,8)$.
And then we may obtain that the non-negative integers which are locally represented by the regular form $\Delta_{m,\mathbf a}(\mathbf x)$ with $a_1=1\eq a_2 \pmod{4},$  $(a_3,a_4) \in \{(4,4),(4,8)\}$, and $a_5 \eq \cdots \eq a_n \eq 0 \pmod{4}$ are
$$\{n \in \N_0| n \not\eq 3 \pmod{4}\}$$
from (\ref{m4}) with $\frac{(m-4)^2}{8(m-2)}(a_1+\cdots+a_n) \in 4\z_2$ because the $2$-adic integers which are represented by the diagonal quadratic forms
$$\Delta_{4,\mathbf a}\otimes \z_2=\left<a_1,\cdots,a_n\right>\otimes \z_2 \supseteq \left<a_1,\cdots,a_4\right>\otimes \z_2$$ 
are $\z_2 \setminus 3+4\z_2$ and $\Delta_{m,\mathbf a}(\mathbf x)$ is $\z_p$-universal for every odd prime $p$.
So we may conclude that a generalized primitive regular $m$-gonal form is of the form of (\ref{4.notuni/4m8}) in this case.
\end{itemize}

{\bf Finally, assume that there are three units of $\z_2$ in $\{a_1,\cdots,a_n\}$ by admitting a recursion.}\\
Without loss of generality, let $a_1\le a_2 \le a_3 \in \z_2 ^{\times}$ and $2 \le a_4\le \cdots \le a_n$ be even.
For our assumption that $\Delta_{m,\mathbf a}(\mathbf x)$ is not universal over $\z_2$, we have that 
\begin{equation} \label{3unit for 2}
a_1\equiv a_2 \equiv a_3 \pmod{4} \text{ and } a_4\equiv \cdots \equiv a_n\equiv 0\pmod{8}.
\end{equation}
Since the diagonal quadratic forms 
$$\left<a_1,\cdots, a_n\right>\otimes \z_2 \supset\left<a_1,a_2,a_3\right> \otimes \z_2$$ represent every prime element of $\z_2$, the generalized $m$-gonal form $\Delta_{m,\mathbf a}(\mathbf x)$ locally represents every non-negative integer in $2+4\z$ by (\ref{m4}).
For the regularity of $\Delta_{m,\mathbf a}(\mathbf x)$ with (\ref{3unit for 2}), we have that 
\begin{equation} \label{3unit for 2'}
(a_1,a_2,a_3,a_4)=(1,1,5,8)
\end{equation} 
for the generalized $m$-gonal form $\Delta_{m,\mathbf a}(\mathbf x)$ locally represents $2,6,10$.
And then we may deduce that the generalized $m$-gonal form $\Delta_{m,\mathbf a}(\mathbf x)$ locally represents every non-negative even integer because the diagonal quadratic forms 
$$\left<a_1,\cdots, a_n\right>\otimes \z_2 \supseteq \left<a_1,a_2,a_3,a_4\right> \otimes \z_2$$ represent every $2$-adic non-unit integer.
Which is a contradiction to the regularity of $\Delta_{m,\mathbf a}(\mathbf x)$ since $\Delta_{m,\mathbf a}(\mathbf x)$ with (\ref{3unit for 2}) and (\ref{3unit for 2'}) does not represent $4$.
So we may conclude that there is no generalized regular $m$-gonal form in this case.
\end{proof}

\vskip 0.8em

\begin{rmk}
Similarly with Remark \ref{rank rmk1}, we may see that there is a generalized regular $m$-gonal form of rank $\ceil{\log_2(m+2)}+1$,  $\ceil{\log_2(m+2)}+2$, and $\ceil{\log_2(m+2)}+1$ of the form of (\ref{1.notuni/4m8}), (\ref{3.notuni/4m8}), and (\ref{4.notuni/4m8}), respectively.
On the other hand, in order for a generalized regular $m$-gonal form $\Delta_{m,\mathbf a}(\mathbf x)$ of the form of (\ref{1.notuni/4m8}), (\ref{3.notuni/4m8}), or (\ref{4.notuni/4m8}) to reprsent every positive integer in $4\N_0$, the rank of $\Delta_{m,\mathbf a}(\mathbf x)$ must be at least $\ceil{\log_2(m-3)}$.
Indeed minimal rank of a generalized regular $m$-gonal form $\Delta_{m,\mathbf a}(\mathbf x)$ of the form of (\ref{1.notuni/4m8}), (\ref{3.notuni/4m8}), and (\ref{4.notuni/4m8}) between $\ceil{\log_2(m-3)}$ and $\ceil{\log_2(m+2)}+1$, $\ceil{\log_2(m-3)}$ and $\ceil{\log_2(m+2)}+2$, and $\ceil{\log_2(m-3)}$ and $\ceil{\log_2(m+2)}+2$, respectively is unknown.
Differently with the case of $m \not\eq 0\pmod{4}$ in Remark \ref{rank rmk1}, there is a chance that the minimal rank of a generalized regular $m$-gonal form is strictly less than the minimal rank of a generalized universal $m$-form for some $m \eq 0 \pmod{4}$.
\end{rmk}

\vskip 1.5em

In Theorem \ref{locuni}, Theorem \ref{notuni3}, and Theorem \ref{notuni2}, we determined every type for generalized primitive regular $m$-gonal form for $m\ge 14$ with $m \not\eq 0\pmod{4}$ and for $m \ge 28$ with $m \eq 8 \pmod{12}$.
They are not $\z_p$-universal for at most one prime $p$.
On the other hand, for $m\ge 28$ with $$m \eq 0\pmod{4} \text{ and } m \not\eq 2\pmod{3},$$ a generalized primitive regular $m$-gonal form could be non-universal over both $\z_2$ and $\z_3$.
In the end, in Theorem \ref{thm m41}, we determine every type for generalized primitive regular $m$-gonal form for $m>28$ with $m\eq 0\pmod{4}$ and $m \not\eq 2\pmod{3}$.

\begin{thm} \label{thm m41}
For $m\ge28$ with $m\eq 0\pmod{4}$ and $m \not\eq 2\pmod{3}$, a generalized primitive regular $m$-gonal form $\Delta_{m,\mathbf a}(\mathbf x)$ is universal or of the form of (\ref{1.notuni/4m8,1m3}), (\ref{3.notuni/4m8,1m3}), (\ref{5.notuni/4m8,1m3}), (\ref{6.notuni/4m8,1m3}), (\ref{7.notuni/4m8,1m3}), (\ref{8.notuni/4m8,1m3}), or (\ref{9.notuni/4m8,1m3}).
\end{thm}
\begin{proof}
We prove the theorem only for $m \equiv4 \pmod{8}$ and $m \equiv1 \pmod{3}$ in this paper.
One may prove the theorem of the other cases (i.e., for $m \eq 16\pmod{24}$, $m \eq 20 \pmod{24}$, or $m \eq 8\pmod{24}$), through similar lengthy processing with this proof. 

For $m\ge 28$ with $m \equiv4 \pmod{8}$ and $m \equiv1 \pmod{3}$, let $\Delta_{m,\mathbf a}(\mathbf x)$ be a generalized primitive regular $m$-gonal form.
By Lemma \ref{lem33}, the generalized regular $m$-gonal form $\Delta_{m,\mathbf a}(\mathbf x)$ is universal over $\z_p$ for every prime $p \not=2,3$.
If $\Delta_{m,\mathbf a}(\mathbf x)$ is universal over $\z_2$ and $\z_3$ too, then the generalized regular $m$-gonal form would be (locally) universal.
Now, we consider a generalized primitive regular $m$-gonal form $\Delta_{m,\mathbf a}(\mathbf x)$ which is not locally universal, i.e., not $\z_2$-universal or not $\z_3$-universal.
If $\Delta_{m,\mathbf a}(\mathbf x)$ is $\z_p$-universal for either $p=2$ or $p=3$, then through same processing with Theorem \ref{notuni3} or Theorem \ref{notuni2}, we may see that the generalized primitive regular $m$-gonal form $\Delta_{m,\mathbf a}(\mathbf x)$ is of the form of (\ref{1.notuni/4m8,1m3}),  (\ref{3.notuni/4m8,1m3}), (\ref{5.notuni/4m8,1m3}), or (\ref{6.notuni/4m8,1m3}).
From now on, we consider generalized primitive regular $m$-gonal form $\Delta_{m,\mathbf a}(\mathbf x)$ which is not universal over both $\z_2$ and $\z_3$.
By Remark \ref{rmk1}, since the $\z_p$-universality of $\Delta_{m,\mathbf a}(\mathbf x)$ coincides with the $\z_p$-universaility of $\Delta_{4,\mathbf a}(\mathbf x)$ for $p=2,3$, we have that there are at most two (resp, three) units of $\z_3$ (resp, $\z_2$) in $\{a_1,\cdots,a_n\}$ by admitting a recursion.
This proof is progressed on each case of the number of units of $\z_2$ and $\z_3$ in $\{a_1,\cdots,a_n\}$.

\vskip 0.8em
{\bf Firstly, we assume that there are one unit of $\z_3$ and one unit of $\z_2$ in $\{a_1,\cdots,a_n\}$ by admitting a recursion.}

Without loss of generality, let $a_1 \in \z_3^{\times}$ and $3 \le a_2\le \cdots \le a_n$ be multiples of $3$.

If $a_1 \in \z_2^{\times}$, then we have that $a_2 \eq \cdots \eq a_n \eq 0 \pmod{6}$ and the generalized regular $m$-gonal form $\Delta_{m,\mathbf a}(\mathbf x)$ (locally) represents every non-negative integer $n$ with $n \eq a_1 \pmod{24}$ because $\Delta_{m,\mathbf a}(\mathbf x)$ represents every $2$-adic (resp, $3$-adic) integer in $a_1+8\z_2$ (resp, $a_1+3\z_3$) over $\z_2$ (resp, over $\z_3$) and is $\z_p$-universal for every $p \not=2,3$.
If $a_1 \not=1$, then for the generalized regular $m$-gonal form $\Delta_{m,\mathbf a}(\mathbf x)$ with $a_2 \eq \cdots \eq a_n \eq 0 \pmod{6}$ (locally) represents $a_1+24, \ a_1+48,$ and $a_1+72(<5(m-3) \le a_1(m-3))$, the generalized $m$-gonal form $a_2P_m(x_2)+\cdots+a_nP_m(x_n)$ with $a_2 \eq \cdots \eq a_n \eq 0 \pmod{6}$ would represent $24,\ 48,$ and $72$.
Which may induce that the generalized $m$-gonal form $\Delta_{m,\mathbf a}(\mathbf x)$ represents every non-negative multiple of $3$ over $\z_3$  from (\ref{m4}) because the diagonal quadratic form
$$\left<a_1,\cdots,a_n\right>\otimes \z_3$$
would represent every non-unit $3$-adic integer.
So we obtain that the integer $n \in [0,23]$ with $n\eq a_1 \pmod{8}$ and $n \eq 0 \pmod{3}$ is locally represented by the generalized $m$-gonal form $\Delta_{m,\mathbf a}(\mathbf x)$. 
Which leads to a contradiction to the regularity of $\Delta_{m,\mathbf a}(\mathbf x)$ because the generalized $m$-gonal form $\Delta_{m,\mathbf a}(\mathbf x)$ with $a_1 \in \z_3^{\times}$ and $a_2 \eq \cdots \eq a_n \eq 0 \pmod{6}$ could not represent the odd multiple of $3$ in $[0,23(<m-3)]$. 

Now consider the case of $a_1=1$.
If $m-3>73$, then we may obtain a contradiction to the regularity of $\Delta_{m,\mathbf a}(\mathbf x)$ through same processing with the above.
It remains to consider the cases of $a_1=1$ with $m-3 \in \{25, 49, 73\}$.
If $m-3 \in \{25, 49, 73\}$, then for the generalized regular $m$-gonal form $\Delta_{m,\mathbf a}(\mathbf x)$ with $a_1=1$ and $a_2 \eq \cdots \eq a_n \eq 0\pmod{6}$ (locally) represents $25$ and $49$, the generalized $m$-gonal form $a_2P_m(x_2)+\cdots+a_nP_m(x_n)$ would represent the prime element $24$ or $48$ (or both of them) of $\z_3$.
Hence there would be at least one prime element of $\z_3$ in $\{a_2, \cdots, a_n\}$.
Which yields that the generalized $m$-gonal form $\Delta_{m,\mathbf a}(\mathbf x)$ represents every non-negative integer $n$ with $n\eq 3\pmod{9}$ or $n \eq 6 \pmod{9}$ over $\z_3$.
So the regular $\Delta_{m,\mathbf a}(\mathbf x)$ would locally represent $33$ or $57$ (that are equivalent with $1=a_1$ modulo $8$).
Which leads to a contradiction to the regularity of the generalized $m$-gonal form $\Delta_{m,\mathbf a}(\mathbf x)$ with $a_1=1$ and $a_2\eq \cdots \eq a_n \eq 0 \pmod{6}$ for $m-3 \in \{25, 49, 73\}$.

Next we assume that $a_1 \in 2\z_2$.
Without loss of generality, let $a_2 \in \z_2^{\times}$ and $6 \le a_3 \le \cdots \le a_n$ be multiples of $6$.
Note that the generalized primitive regular $m$-gonal form $\Delta_{m,\mathbf a}(\mathbf x)$ (locally) represents every non-negative integer $n$ with $n \eq a_1 \pmod{3}$ and $n \eq a_2 \pmod{8}$.
Since there is an integer $n_1 \in [1,23]$ for which $n_1 \eq a_1 \pmod{3}$ and $n_1 \eq a_2 \pmod{8}$, we have that $a_1,a_2 \le n_1 <24$ for the regularity of $\Delta_{m,\mathbf a}(\mathbf x)$.

First suppose that $a_1 \not\eq 0 \pmod{8}$.
Since $n_1 \not=a_1+a_2$, for the generalized regular $m$-gonal form $\Delta_{m,\mathbf a}(\mathbf x)$ (locally) represents $n_1(\le 23)$ and $n_1+24(\le 47)$, $a_3$ and $a_4$ could not exceed $23$ and $47$, respectively.
So we have that $\ord_2(a_3) \le 2$ and $\ord_2(a_4) \le 3$.
If $\ord_2(a_i)=1$ for some $i \not= 2$, then the diagonal quadratic form 
$$\left<a_1,\cdots,a_n\right> \otimes \z_2 = \Delta_{4,\mathbf a}\otimes \z_2$$
would represent every non-unit of $\z_2$, which yields that the generalized $m$-gonal form $\Delta_{m,\mathbf a}(\mathbf x)$ represents every non-negative even integer over $\z_2$.
The generalized regular $m$-gonal form $\Delta_{m,\mathbf a}(\mathbf x)$ would (locally) represent every non-negative integer in
$\{2,8,14,20,\cdots\}$ or $\{4,10,16,22,\cdots\}$ depending on the residue of $a_1$ modulo $3$.
Therefore we have that $a_1 \in \{2,4\}$, $a_3 \in \{6\}$, and $a_4 \in \{6, 12\}$, which yields that the generalized $m$-gonal form $\Delta_{m,\mathbf a}(\mathbf x)$ is $\z_2$-univesal.
For our assumption, we have that $a_i \eq 0 \pmod{4}$ for all $i \not=2$ and then since $\ord_2(a_1)=2$, $\ord_2(a_2) = 0$, $\ord_2(a_3) = 2$, and $\ord_2(a_4) \le 3$,
we may obtain that the generalized $m$-gonal form $\Delta_{m,\mathbf a}(\mathbf x)$ represents every non-negative multiple of $4$ over $\z_2$.
Which yields that the generalized regular $m$-gonal form $\Delta_{m,\mathbf a}(\mathbf x)$ (locally) represents every non-negative integer $n$ which is equivalent with $a_1$ modulo $12$.
For the regularity of $\Delta_{m,\mathbf a}(\mathbf x)$, we obtain that $a_1 \eq 4 \pmod{8}$ is less than $12$, yielding $a_1=4$. 
And we have that $a_2 \eq 3 \pmod{4}$ because otherwise, the generalized $m$-gonal form $\Delta_{m,\mathbf a}(\mathbf x)$ locally represents $1$.
Since the generalized regular $m$-gonal form $\Delta_{m,\mathbf a}(\mathbf x)$ with $a_1=4$, $a_2 \eq 3 \pmod{4}$, and $a_3\eq \cdots \eq a_n \eq 0 \pmod{12}$ (locally) represents $a_1, \ a_1+12, \ a_1+24$, we have that $a_3=12$ and $a_4 \in \{12, 24\}$.
Which yields that the generalized $m$-gonal form $\Delta_{m,\mathbf a}(\mathbf x)$ represents every non-negative multiple of $3$ over $\z_3$ because $\ord_3(a_1)=0$ and $a_2<27$, i.e., $1 \le \ord_3(a_2) \le 2$.
Since the generalized regular $m$-gonal form $\Delta_{m,\mathbf a}(\mathbf x)$ (locally) represents $3$, we have that $a_2=3$.
In sum the results in this case, the generalized regular $m$-gonal form $\Delta_{m,\mathbf a}(\mathbf x)$ have $a_1=4$, $a_2=3$, $a_3=12$, $a_4 \in \{12,24\}$, and $a_5 \eq \cdots \eq a_n \eq 0 \pmod{12}$.
Therefore the generalized regular $m$-gonal form $\Delta_{m,\mathbf a}(\mathbf x)$ (locally) represents every non-negative integer which is equivalent with $0, \ 3, \ 4,$ or $7$ modulo $12$ and that's all because the $2$-adic integers which are represented by 
$$\left<a_1,\cdots,a_n\right> \otimes \z_2 = \Delta_{4,\mathbf a}\otimes \z_2$$
are $a_2+4\z_2 \cup 2\z_2$, the $3$-adic integers which are represented by
$$\left<a_1,\cdots,a_n\right> \otimes \z_3 = \Delta_{4,\mathbf a}\otimes \z_3$$
are $a_1+3\z_3 \cup 3\z_3$, and
$\Delta_{m,\mathbf a}(\mathbf x)$ is $\z_p$-universal for all $ p\not= 2,3$.
Consequently, we may conclude that a generalized primitive regular $m$-gonal form is of the form of (\ref{7.notuni/4m8,1m3}) in this case.

If $a_1\eq 0 \pmod{8}$, then $n_1=a_1+a_2$.
For the generalized regular $m$-gonal form $\Delta_{m,\mathbf a}(\mathbf x)$ with $a_3\eq \cdots \eq a_n \eq 0\pmod{6}$ (locally) represents $n_1+24$ and $n_1+48$, the generalized $m$-gonal form $a_3P_m(x_3)+\cdots+a_nP_m(x_n)$ 
would represent $24$ and $48$.
Which may yield that the generalized regular $m$-gonal form $\Delta_{m,\mathbf a}(\mathbf x)$ (with $a_1<24$ and $a_2\eq 1\pmod{2}$) represents every non-negative multiple of $4$ over $\z_2$,
so we obtain that the generalized regular $m$-gonal form $\Delta_{m,\mathbf a}(\mathbf x)$ (locally) represents every non-negative integer $n$ with $n \eq a_1 \pmod{12}$.
For the regularity of $\Delta_{m,\mathbf a}(\mathbf x)$, we have that $a_1=8$ since $a_1 \le 12$.
For the generalized regular $m$-gonal form $\Delta_{m,\mathbf a}(\mathbf x)$ with $a_1=8, \ a_2 \eq 3 \pmod{6}$, and $a_3\eq \cdots \eq a_n \eq 0 \pmod{6}$ (locally) represents $a_1, \ a_1+12,\ a_1+24$, and $a_1+36$, the generalized $m$-gonal form $a_3P_m(x_3)+\cdots+a_nP_m(x_n)$ represents $12, \ 24,$ and $36$.
If $a_i \eq 2 \pmod{4}$ for some $i$, then the generalized $m$-gonal form $\Delta_{m,\mathbf a}(\mathbf x)$ represents every non-negative even integer over $\z_2$, which leads to a contradiction to the regularity of $\Delta_{m,\mathbf a}(\mathbf x)$ because $\Delta_{m,\mathbf a}(\mathbf x)$ could not represent $2$ which is locally represented by $\Delta_{m,\mathbf a}(\mathbf x)$.
So we have that $a_3 \eq \cdots \eq a_n \eq 0 \pmod{12}$, yielding $a_3=12$, $a_4 \in \{12,24\}$.
And we have that $a_2 \eq 3\pmod{4}$, because otherwise, the generalized $m$-gonal form $\Delta_{m,\mathbf a}(\mathbf x)$ locally represents $5$.
The generalized regular $m$-gonal form $\Delta_{m,\mathbf a}(\mathbf x)$ with $a_1=8$, $a_2\eq 3 \pmod{4}$, $a_3=12$, and $a_4 \in \{12,24\}$ (locally) represents $3$, which yields that $a_2=3$.
In sum the results in this case, the generalized regular $m$-gonal form $\Delta_{m,\mathbf a}(\mathbf x)$ have $a_1=8$, $a_2=3$, $a_3=12$, $a_4 \in \{12,24\}$, and $a_5 \eq \cdots \eq a_n \eq 0 \pmod{12}$.
Therefore the generalized regular $m$-goanl form $\Delta_{m,\mathbf a}(\mathbf x)$ (locally) represents every non-negative integer which is equivalent with $0, \ 3, \ 8,$ or $11$ modulo $12$ and that's all because the $2$-adic integers which are represented by
$$\left<a_1,\cdots,a_n\right> \otimes \z_2 = \Delta_{4,\mathbf a}\otimes \z_2$$
are $a_2+4\z_2 \cup 2\z_2$, the $3$-adic integers which are represented by 
$$\left<a_1,\cdots,a_n\right> \otimes \z_3 = \Delta_{4,\mathbf a}\otimes \z_3$$
are $a_1+3\z_3 \cup 3\z_3$, and 
$\Delta_{m,\mathbf a}(\mathbf x)$ is $\z_p$-universal for all $ p\not= 2,3$.
Consequently, we may conclude that the generalized primitive regular $m$-gonal forms are of the form of (\ref{8.notuni/4m8,1m3}) in this case.


\vskip 0.8em
{\bf Secondly, we assume that there are one unit of $\z_3$ and two units of $\z_2$ in $\{a_1,\cdots,a_n\}$ by admitting a recursion.}

Without loss of generality, let $a_1 \in \z_3^{\times}$ and $3 \le a_2\le \cdots \le a_n$ be multiples of $3$.

First we suppose that $\left<a_1,\cdots,a_n\right>\otimes \z_2 \supset \left<u,u'\right>\otimes \z_2$ for some $u \equiv 1 \text{ and } u'\equiv3 \pmod{4}$.
Since the generalized $m$-gonal form $\Delta_{m,\mathbf a}(\mathbf x)$ represents every non-negative odd integer over $\z_2$, the generalized $m$-gonal form $\Delta_{m,\mathbf a}(\mathbf x)$ would locally represent every integer in $\{1,7,13,19,\cdots\}$ (or resp, $\{5,11,17,23, \cdots\}$) when $a_1 \eq 1\pmod{3}$ (or resp, $a_1 \eq 2 \pmod{3}$).
Hence for the regularity of $\Delta_{m,\mathbf a}(\mathbf x)$, we have that $a_1 \in \{1,2,5\}$ and the generalized $m$-gonal form $a_2P_m(x_2)+\cdots+a_nP_m(x_n)$ with $a_2 \eq \cdots \eq a_n \eq 0 \pmod{3}$ represents $6,12,$ and $18$ (or $3,9,15,$ and $21$).
Which yields that the generalized regular $m$-gonal form $\Delta_{m,\mathbf a}(\mathbf x)$ represents every non-negative multiple of $3$ over $\z_3$ and so (locally) represents every non-negative odd  multiple of $3$, i.e., $3,9,15, \cdots$.
From this, we may induce that the generalized $m$-gonal form $\Delta_{m,\mathbf a}(\mathbf x)$ is $\z_2$-universal.
Therefore in this case, there is no generalized regular $m$-gonal form.

Next we suppose that $\left<a_1,\cdots,a_n\right>\otimes \z_2 \not\supset \left<u,u'\right>\otimes \z_2$ for any $u \equiv 1 \text{ and } u'\equiv3 \pmod{4}$.
In this case, we have that $a_i \not\eq 2 \pmod{4}$ and $a_1' \eq a_2' \pmod{4}$ for two units $a_1' \le a_2'$ of $\z_2$ in $\{a_1,\cdots, a_n\}$.
The generalized $m$-gonal form $\Delta_{m,\mathbf a}(\mathbf x)$ would represent every non-negative integer which is equivalent with $a_1'$ modulo $4$ over $\z_2$.\\
If $a_1 \not\in \z_2^{\times}$, then since the generalized regular $m$-gonal form $\Delta_{m,\mathbf a}(\mathbf x)$ (locally) represents every non-negative integer which is equivalent with $n_1:=a_1+a_1'$ modulo $12$, $a_1+a_1'$ would be less than $12$.
Without loss of generality, let $a_1'=a_2$ and $a_2'=a_3$, then we have that $a_4 \le \cdots \le a_n $ are multiples of $12$.
For the generalized regular $m$-gonal form $\Delta_{m,\mathbf a}(\mathbf x)$ (locally) represents $n_1+12, \ n_1+24, \ n_1+36 , \ n_1+48,$ and $n_1+60$, 
the generalized $m$-gonal form $a_4P_m(x_4)+\cdots+a_nP_m(x_n)$ with $a_4\eq \cdots \eq a_n \eq 0 \pmod{12}$ would represent at least three integers among $12, \ 24, \ 36, \ 48, $ and $60$.
Which may yield that the generalized $m$-gonal form $\Delta_{m,\mathbf a}(\mathbf x)$ represents every $2$-adic integer in $a_2+4\z_2 \cup 2\z_2$ over $\z_2$ and every $3$-adic integer in $a_1+3\z_3 \cup 3 \z_3$ over $\z_3$. 
Hence the generalized regular $m$-gonal form $\Delta_{m,\mathbf a}(\mathbf x)$ would (locally) represent $2$ or $4$, resp, if $a_1 \eq 2 \pmod{3}$ or $a_1 \eq 1 \pmod{3}$, resp.
For our assumption and regularity of $\Delta_{m,\mathbf a}(\mathbf x)$, we have that $a_1=4$ and $a_2 \eq a_3 \eq 3\pmod{4}$.
Since the generalized regular $m$-gonal form $\Delta_{m,\mathbf a}(\mathbf x)$ (locally) represents $3$ and $6$, we have that $a_2=a_3=3$.
In sum the results in this case, the generalized regular $m$-gonal form $\Delta_{m,\mathbf a}(\mathbf x)$ have $a_1=4$, $a_2=3$, $a_3=3$, $a_4=12$, and $a_5 \eq \cdots \eq a_n \eq 0 \pmod{12}$.
Therefore the generalized regular $m$-gonal form $\Delta_{m,\mathbf a}(\mathbf x)$ (locally) represents every non-negative integer which is equivalent with $0, \ 3, \ 4, \ 6, \ 7,$ or $10$ modulo $12$ and that's all because the $2$-adic integers which are represented by
$$\left<a_1,\cdots,a_n\right> \otimes \z_2 = \Delta_{4,\mathbf a}\otimes \z_2$$
are $a_2+4\z_2 \cup 2\z_2$, the $3$-adic integers which are represented by 
$$\left<a_1,\cdots,a_n\right> \otimes \z_3 = \Delta_{4,\mathbf a}\otimes \z_3$$
are $a_1+3\z_3 \cup 3\z_3$, and
$\Delta_{m,\mathbf a}(\mathbf x)$ is $\z_p$-universal for all $ p\not= 2,3$.
Consequently, we may conclude that the generalized primitive regular $m$-gonal forms are of the form of (\ref{9.notuni/4m8,1m3}) in this case. \\
If $a_1  \in \z_2^{\times}$, then the generalized regular $m$-gonal form $\Delta_{m,\mathbf a}(\mathbf x)$ (locally) represents every non-negative integer which is equivalent with $a_1$ modulo $12$.
So $a_1$ would be less than $12$.
Without loss of generality, let $a_2$ be an odd integer and $a_3 \le  \cdots \le a_n $ be multiples of $12$.
For the generalized regular $m$-gonal form $\Delta_{m,\mathbf a}(\mathbf x)$ (locally) represents $a_1+12$, we have that $a_3=12$.
Since the generalized $m$-gonal form $\Delta_{m,\mathbf a}(\mathbf x)$ with $a_1\eq a_2 \pmod{4}$ and $a_3=12$ represents every prime element of $\z_2$ over $\z_2$, the generalized $m$-gonal form $\Delta_{m,\mathbf a}(\mathbf x)$ locally represents $2$ (or resp, $10$) if $a_1 \eq 2\pmod{3}$ (or resp, $a_1 \eq 1\pmod{3}$).
For our assumptions and the regularity of $\Delta_{m,\mathbf a}(\mathbf x)$ with $a_2 \eq \cdots \eq a_n \eq 0 \pmod{3}$, since $\Delta_{m,\mathbf a}(\mathbf x)$ could not represent $2$, we have that $a_1 \eq 1 \pmod{3}$ and $a_1=1$ and $a_2=9$.
Which leads to a contradiction to the regularity of $\Delta_{m,\mathbf a}(\mathbf x)$ because $\Delta_{m,\mathbf a}(\mathbf x)$ could not represent $4$ which is locally represented by $\Delta_{m,\mathbf a}(\mathbf x)$.
Thus in this case, there is no generalized regular $m$-gonal form.

\vskip 0.8em
{\bf Thirdly, we assume that there are one unit of $\z_3$ and three units of $\z_2$ in $\{a_1,\cdots,a_n\}$ by admitting a recursion.}

Without loss of generality, let $a_1\le a_2\le a_3 \in \z_2^{\times}$ and $2 \le a_4\le \cdots \le a_n$ be multiples of $2$.
For our assumption that the generalized $m$-gonal form $\Delta_{m,\mathbf a}(\mathbf x)$ is not $\z_2$-universal, we have that 
\begin{equation}\label{13}a_1\eq a_2 \eq a_3 \pmod{4} \text{ and } a_4 \eq \cdots \eq a_n \eq 0 \pmod{8}.\end{equation}
Since the generalized $m$-gonal form $\Delta_{m,\mathbf a}(\mathbf x)$ represents every prime element in $\z_2$ over $\z_2$, the generalized $m$-gonal form $\Delta_{m,\mathbf a}(\mathbf x)$ would locally represent $2$ or $10$ depending on the residue of the unit of $\z_3$ in $\{a_1,\cdots, a_n\}$.
Since under our assumptions, the generalized $m$-gonal form $\Delta_{m,\mathbf a}(\mathbf x)$ could not represent $2$, the generalized regular $m$-gonal form $\Delta_{m,\mathbf a}(\mathbf x)$ would represent $10$.
For the generalized $m$-gonal form $\Delta_{m,\mathbf a}(\mathbf x)$ with (\ref{13}) represent $10$, $(a_1,a_2,a_3)$ should be $(1,9,a_3)$, $(3,7,a_3)$ or $(3,3,7)$.
The generalized $m$-gonal form $\Delta_{m,\mathbf a}(\mathbf x)$ with $(a_1,a_2,a_3)=(1,9,a_3)$ (or resp, $(a_1,a_2,a_3) \in \{ (3,3,7),(3,7,a_3) \}$) locally represents $13$ (or resp, $19$), which leads to a contradiction to the regularity of $\Delta_{m,\mathbf a}(\mathbf x)$ with (\ref{13}).
Thus, in this case, there is no generalized regular $m$-gonal form.
\vskip 0.8em

{\bf Fourthly, we assume that there are two units of $\z_3$ and one unit of $\z_2$ in $\{a_1,\cdots,a_n\}$ by admitting a recursion.}

Without loss of generality, let $a_1\le a_2 \in \z_3^{\times}$ and $3 \le a_3\le \cdots \le a_n$ be multiples of $3$.
Since the generalized regular $m$-gonal form $\Delta_{m,\mathbf a}(\mathbf x)$ is not $\z_3$-universal, we have that $a_1\eq a_2 \pmod{3}$ from Remark \ref{rmk1}.
For the unit $a_1'$ of $\z_2$ in $\{a_1,\cdots, a_n\}$, the generalized $m$-gonal form $\Delta_{m,\mathbf a}(\mathbf x)$ locally represents every non-negative integer $n$ with $n \eq 2a_1 \pmod{3}$ and $n \eq a_1' \pmod{8}$.
Namely, for the integer $n_1 \in [0,23]$  with $n_1 \eq 2a_1 \pmod{3}$ and $n_1 \eq a_1' \pmod{8}$, the generalized regular $m$-gonal form $\Delta_{m,\mathbf a}(\mathbf x)$ would locally represents every non-negative integer $n$ with $n \eq n_1 \pmod{24}$.
When $(m;a_1)\not=(28;1), (28;2)$ and $(52;1)$, in order for the generalized $m$-gonal form $\Delta_{m,\mathbf a}(\mathbf x)$ with $a_3 \eq \cdots \eq a_n \eq 0 \pmod{3}$ to represent $n_1, \ n_1+24$, and $n_1+48$, there needs at least two prime elements of $\z_3$ in $\{a_3,\cdots ,a_n\}$.
Which leads to a contradiction to our assumption that the generalized $m$-gonal form $\Delta_{m,\mathbf a}(\mathbf x)$ is not $\z_3$-universal.\\
Now we consider the cases of $(m;a_1) \in \{(28;1),(28;2),(52;1)\}$.
If $(m;a_1)=(52;1)$, then since the generalized $m$-gonal form $\Delta_{m,\mathbf a}(\mathbf x)$ with $a_2 \eq 0 \pmod{2}$ and $a_3\eq \cdots \eq a_n \eq 0\pmod{6}$ locally represents $17$, we have that $a_2<17$, i.e., $a_2 \in \{4,10,16\}$ for the regularity of $\Delta_{m,\mathbf a}(\mathbf x)$.
More specifically, the candidates for $(a_2,a_3)$ are $(4,12) , \ (10, 6),$ and $(16,a_3)$.
Since the generalized $m$-gonal form $\Delta_{m,\mathbf a}(\mathbf x)$ with $(a_1,a_2,a_3)=(1,4,12)$ (resp, $(a_1,a_2,a_3)=(1,10,6)$) locally represents $8$ (resp, $2$), for the regularity of $\Delta_{m,\mathbf a}(\mathbf x)$, we have that $a_2=16$.
Since the generalized $m$-gonal form $\Delta_{m,\mathbf a}(\mathbf x)$ with $(a_1,a_2)=(1,16)$ locally represents $25$, the generalized $m$-gonal form $a_3P_m(x_3)+\cdots+a_nP_m(x_n)$ should represent $24$ for the regularity of $\Delta_{m,\mathbf a}(\mathbf x)$.
Because the diagonal quadratic form $\left<a_1,a_2,24\right>\otimes \z_2=\left<1,16,24\right>\otimes \z_2$ represents every $2$-adic integer in $4\z_2^{\times}$, 
the generalized $m$-gonal form $\Delta_{m,\mathbf a}(\mathbf x)$ locally represents $8$ which is a contradiction to the regularity of $\Delta_{m,\mathbf a}(\mathbf x)$ with $(a_1,a_2)=(1,16)$ and $a_3 \eq \cdots \eq a_n \eq 0\pmod{6}$.\\
If $(m;a_1)=(28;1)$, then since the generalized $m$-gonal form $\Delta_{m,\mathbf a}(\mathbf x)$ locally represents $17$, similarly with the case of $(m;a_1)=(52;1)$, for the regularity of the generalized $m$-gonal form $\Delta_{m,\mathbf a}(\mathbf x)$ with $a_2 \eq 0 \pmod{2}$ and $a_3\eq \cdots \eq a_n \eq 0\pmod{6}$, we have that $a_2=16$.
In order for the generalized $m$-gonal form $\Delta_{m,\mathbf a}(\mathbf x)$ with $a_1=1$, $a_2 =16$ and $a_3\eq \cdots \eq a_n \eq 0\pmod{6}$ to represent $1,25,49,73,$ and $97$ which are locally represented by $\Delta_{m,\mathbf a}(\mathbf x)$, $a_3$ and $a_4$ could not exceed $48$ and $96$, respectively, so we have that $\ord_3(a_3)\le 2$ and $\ord_3(a_4)\le 3$.
Which leads to a contradiction to the regularity of $\Delta_{m,\mathbf a}(\mathbf x)$ with $(a_1,a_2)=(1,16)$ and $a_3\eq \cdots \eq a_n \eq 0\pmod{6}$ because the generalized $m$-gonal form $\Delta_{m,\mathbf a}(\mathbf x)$ with $(a_1,a_2)=(1,16)$, $\ord_3(a_3)\le 2$ and $\ord_3(a_4)\le 3$ locally represents $9$.\\
If $(m;a_1)=(28;2)$, then since the generalized regular $m$-gonal form $\Delta_{m,\mathbf a}(\mathbf x)$ with $2=a_1\eq a_2 \pmod{3}$ and $a_3 \eq \cdots \eq a_n \eq 0 \pmod{3}$ (locally) represents the integer $n \in [0,23]$ with $n \eq 1 \pmod{3}$ and $n \eq a_i \pmod{8}$ for the unique odd coefficient $a_i$ in $\{a_1,\cdots,a_n\}$, we have that $a_j \eq -2 \pmod{8}$ for some $j \not=1$ because $\Delta_{m,\mathbf a}(\mathbf x)$ is not $\z_2$-universal.
For the regularity of the $\Delta_{m,\mathbf a}(\mathbf x)$, we have that $\Delta_{m,\mathbf a}(\mathbf x)$ does not represent $1$ over $\z_2$.
Hence the generalized $m$-gonal form $\Delta_{m,\mathbf a}(\mathbf x)$ with $a_1=2$, $a_i \eq 1 \pmod{2}$, and $a_j \eq -2 \pmod{8}$ would represent every non-negative odd integer $n$ with $n \not\eq 1 \pmod{8}$ and even integer $n$ with $n \eq 2 \pmod{4}$ over $\z_2$.
So we obtain that the generalized $m$-gonal form $\Delta_{m,\mathbf a}(\mathbf x)$ locally represents $5,7,$ and $10$.
On the other hand, in order for the generalized $m$-gonal from $\Delta_{m,\mathbf a}(\mathbf x)$ having only one odd coefficient in $\{a_1,\cdots,a_n\}$ with $2=a_1 \eq a_2 \pmod{3}$ and $a_3 \eq \cdots \eq a_n \eq 0\pmod{3}$ to represent $5,7,$ and $10$, $(a_1,a_2,a_3,a_4)$ should be $(2,2,3,6)$ and the generalized $m$-gonal form $\Delta_{m,\mathbf a}(\mathbf x)$ with $(a_1,a_2,a_3,a_4) = (2,2,3,6) $ is $\z_2$-universal.
So we may conclude that in this case, there is no generalized regular $m$-gonal form.

\vskip 0.8em

{\bf Finally, we assume that there are two units of $\z_3$ and two or three units of $\z_2$ in $\{a_1,\cdots,a_n\}$ by admitting a recursion.}

Without loss of generality, let $a_1\le a_2 \in \z_3^{\times}$ and $3 \le a_3\le \cdots \le a_n$ be multiples of $3$.
Since the generalized regular $m$-gonal form $\Delta_{m,\mathbf a}(\mathbf x)$ is not $\z_3$-universal, we have that $a_1\eq a_2 \pmod{3}$ from Remark \ref{rmk1}.
Since there are at least two units of $\z_2$ in $\{a_1,\cdots, a_n\}$, the generalized regular $m$-gonal form $\Delta_{m,\mathbf a}(\mathbf x)$ (locally) represents every non-negative integer $n$ with $n \eq 2a_1 \pmod{3}$ and $n \eq a_i \pmod{4}$ for an odd coefficient $a_i$ in $\{a_1,\cdots,a_n\}$.
Namely, for the integer $n_1 \in [0,11]$  with $n_1 \eq 2a_1 \pmod{3}$ and $n_1 \eq a_i \pmod{4}$, the generalized regular $m$-gonal form $\Delta_{m,\mathbf a}(\mathbf x)$ (locally) represents every non-negative integer $n$ with $n \eq n_1 \pmod{12}$.
When $(m;a_1)\not=(28;1)$, in order for the generalized $m$-gonal form $\Delta_{m,\mathbf a}(\mathbf x)$ with $a_3 \eq \cdots \eq a_n \eq 0 \pmod{3}$ to represent $n_1, \ n_1+12$, and $n_1+24$, there needs at least two prime elements of $\z_3$ in $\{a_3,\cdots ,a_n\}$.
Which implies that the generalized $m$-gonal form $\Delta_{m,\mathbf a}(\mathbf x)$ is $\z_3$-universal.
If $(m;a_1) = (28;1)$, then since the generalized $m$-gonal form $\Delta_{m,\mathbf a}(\mathbf x)$ (which is not $\z_2$-universal) locally represents $5$, we have that $(a_1,a_2)=(1,4)$ for the regularity of $\Delta_{m,\mathbf a}(\mathbf x)$ with $1=a_1 \eq a_2 \pmod{3}$.
The generalized $m$-gonal form $\Delta_{m,\mathbf a}(\mathbf x)$ with $(a_1,a_2)=(1,4)$ locally represents $11$ and $17$ (or resp, $2$) if there is an odd $a_i$ which is equivalent to $3$ modulo $4$ (or resp, otherwise).
The generalized  $m$-gonal form $\Delta_{m,\mathbf a}(\mathbf x)$ with $(a_1,a_2)=(1,4)$ and $a_3 \eq \cdots \eq a_n \eq 0 \pmod{3}$ could not represent $2$.
And in order for the generalized $m$-gonal form $\Delta_{m,\mathbf a}(\mathbf x)$ with $(a_1,a_2)=(1,4)$ and $a_3 \eq \cdots \eq a_n \eq 0 \pmod{3}$, to represent $11$ and $17$, there needs at least two prime elements of $\z_3$ in $\{a_3,\cdots, a_n\}$, which yields that $\Delta_{m,\mathbf a}(\mathbf x)$ is $\z_3$-universal.
Consequently, we may conclude that there is no generalized regular $m$-gonal form in this case.
This completes the proof.
\end{proof}

\begin{rmk}
Through Theorem \ref{locuni}, Theorem \ref{notuni3}, Theorem \ref{notuni2}, and Theorem \ref{thm m41}, we determined every type for generalized primitive regular $m$-gonal form for 
$m \ge 14$ with $m \not\eq0 \pmod{4}$ and $m \ge 28$ with $m \eq 0\pmod{4}$ on the local properties of Lemma \ref{lem33}.
On the other hand, a generalized primitive regular shifted $m$-gonal also satisfy the local properties.
Through similar processings with Theorem \ref{notuni3}, when $m\ge14$ is an odd integer with $m \not\eq 2 \pmod{3}$, one may obtain a regular generalized shifted $m$-gonal form which is not generalized $m$-gonal form
$$\left<1^{(2)},3a_2^{(r_2)},\cdots, 3a_n^{(r_n)}\right>_m$$
which represents every non-negative integer $n$ with $n \not\eq 1 \pmod{3}$.
Like this, one may determine every type for generalized primitive regular shifted $m$-gonal form for $m \ge 14$ with $m \not\eq0 \pmod{4}$ and $m \ge 28$ with $m \eq 0\pmod{4}$, too through modification of Theorem \ref{locuni}, Theorem \ref{notuni3}, Theorem \ref{notuni2} and Theorem \ref{thm m41} or different way.
\end{rmk}

\begin{rmk}
In the end, we consider a regular $m$-gonal form (not generalized $m$-gonal form), i.e., which admits its variables only non-negative integers.
For this, we define {\it $x$-th shifted $m$-gonal number of level $r$} as
\begin{equation}P_m^{(r)}(x):=\frac{m-2}{2}x^2-\frac{m-2-2r}{2}x \end{equation}
for $x \in \N_0$ and a positive integer $r$ with $(r,m-2)=1$.
And we call a weighted sum of shifted $m$-gonal numbers (i.e., which admits only non-negative variables $x_i \in \N_0$)
\begin{equation}\Delta_{m,\mathbf a}^{(\mathbf r)}(\mathbf x):=a_1P_m^{(r_1)}(x_1)+\cdots+a_nP_m^{(r_n)}(x_n)\end{equation} where $\mathbf a \in \N^n$ as {\it shifted $m$-gonal form}.

\begin{prop} \label{prop2}
Let $\Delta_{m,\mathbf a}^{(\mathbf r)}(\mathbf x)$ be a primitive shifted $m$-gonal form.
\begin{itemize}
\item[(1)] $\Delta_{m,\mathbf a}^{(\mathbf r)}(\mathbf x)$ is $\z_p$-universal for any prime $p$ with $(p,2(m-2)) \not= 1$.
\item[(2)] When $m \not\eq 0\pmod{4}$, $\Delta_{m,\mathbf a}^{(\mathbf r)}(\mathbf x)$ is $\z_2$-universal.
\end{itemize}
\end{prop}
\begin{proof}
This proof is same with the proof of Proposition \ref{prop1}.
\end{proof}

For a shifted $m$-gonal form $\Delta_{m,\mathbf a}^{(\mathbf r)}(\mathbf x)$ and an odd prime $p$ with $(p,(m-2))=1$,
we define a set
$$\Lambda_p(\Delta_{m,\mathbf a}^{(\mathbf r)}):=\{\mathbf x \in \N_0^n|\Delta_{m,\mathbf a}^{(\mathbf r)}(\mathbf x+\mathbf y) \eq \Delta_{m,\mathbf a}^{(\mathbf r)}(\mathbf x-\mathbf y) \pmod{p} \text{ for all } \mathbf y \in \N_0^n\}$$
and when $m \eq 0\pmod{4}$, we define a set
$$\Lambda_2(\Delta_{m,\mathbf a}^{(\mathbf r)}):=\{\mathbf x \in \N_0^n|\Delta_{m,\mathbf a}^{(\mathbf r)}(\mathbf x+\mathbf y) \eq \Delta_{m,\mathbf a}^{(\mathbf r)}(\mathbf x-\mathbf y) \pmod{8} \text{ for all } \mathbf y \in \N_0^n\}.$$
Then we have that
$$\Lambda_p(\Delta_{m,\mathbf a}^{(\mathbf r)})=\left\{(\ell(x_1),\cdots,\ell(x_n))|x_i \in \N_0\right\}$$
where $$\ell(x_i)=\begin{cases}px_i+j_i & \text{if } a_i \in \z_p^{\times} \\ x_i & \text{if } a_i \in p \z_p \end{cases}$$ 
for the integer $j_i$ satisfying 
$$\begin{cases}
0 \le j_i \le p-1 \\
j_i \eq \frac{m-2-2r_i}{2(m-2)} \pmod{p}.
\end{cases}$$
Without loss of generality, let $a_1,\cdots,a_t \in \z_p^{\times}$ and $a_{t+1},\cdots,a_n \in p\z_p$, we may see that for $(\ell(x_1),\cdots,\ell(x_n)) \in \Lambda_p(\Delta^{(\mathbf r)}_{m,\mathbf a})$,
\begin{equation}\label{lambdaeq;p''}
\Delta_{m,\mathbf a}^{(\mathbf r)}(\ell(x_1),\cdots,\ell(x_n))=\sum_{i=1}^tp^2a_iP_m^{(r_i')}(x_i)+\sum_{i=t+1}^na_iP_m^{(r_i)}(x_i)+C
\end{equation}
where $r_i'=\frac{(m-2)(p+2j_i)-(m-2-2r_i)}{2p}$ 
and $C=a_1P_m^{(r_1)}(j_1)+\cdots +a_tP_m^{(r_t)}(j_t)$.
Since $r_i'$ is a positive integer which is relatively prime with $m-2$, the quadratic polynomial $\sum_{i=1}^tp^2a_iP_m^{(r_i')}(x_i)+\sum_{i=t+1}^na_iP_m^{(r_i)}(x_i)$ in (\ref{lambdaeq;p''}) is also a shifted $m$-gonal form.
We define the shifted $m$-gonal form in (\ref{lambdaeq;p''}) after 
scaling for some suitable rational integer for its primitivity as
\begin{equation}\label{lambdaeq;p}
\lambda_p\left(\Delta_{m,\mathbf a}^{(\mathbf r)} \right)(\mathbf x):=\frac{1}{p^k}\left(\sum_{i=1}^tp^2a_iP_m^{(r_i')}(x_i)+\sum_{i=t+1}^na_iP_m^{(r_i)}(x_i)\right)
\end{equation}
where $k=\min\{2,\ord_p(a_{t+1}),\cdots,\ord_p(a_{n})\}$.

\vskip1em

And then we may obtain the following lemma similarly with Lemma \ref{lambda is regular}.
\begin{lem} \label{lambda is regular'}
Let $\Delta_{m,\mathbf a}^{(\mathbf r)}(\mathbf x)$ be a primitve regular shifted $m$-gonal form. 
\begin{itemize}
\item[(1)] If $\Delta_{m,\mathbf a}^{(\mathbf r)}(\mathbf x)$ is not $\z_p$-universal for some odd prime $p$, then its $\lambda_p$-transformation
$\lambda_p (\Delta_{m,\mathbf a}^{(\mathbf r)} )(\mathbf x)$ is regular.
\item[(2)] For $m \eq 0 \pmod{4}$, if $\left<a_1,\cdots,a_n\right>\otimes \z_2 \not\supset \left<u,u'\right>\otimes \z_2$ for any $u \eq 1 \pmod{3}$ and $u' \eq 3 \pmod{4}$, then its $\lambda_2$-transformation
$\lambda_2 (\Delta_{m,\mathbf a}^{(\mathbf r)} )(\mathbf x)$ is regular.
\end{itemize}
\end{lem}

Also through a similar processing with Lemma \ref{lem33}, we may obatin the following lemma.

\begin{lem} 
Let $\Delta_{m,\mathbf a}^{(\mathbf r)}(\mathbf x)$ be a primitve regular shifted $m$-gonal form. 
\begin{itemize}
\item[(1)] For $m\ge9$ with $m \not\equiv 0 \pmod{4}$, $\Delta_{m,\mathbf a}^{(\mathbf r)}(\mathbf x)$ of rank $n\ge4$ is $\z_p$-universal for every prime $p \not=3$.
\item[(2)] For $m\ge 12$ with $m \equiv 0 \pmod{4}$, $\Delta_{m,\mathbf a}^{(\mathbf r)}(\mathbf x)$ of rank $n \ge 4$ is $\z_p$-universal for every prime $p \not=2,3$.
\end{itemize}
\end{lem}

Finally, we may determine every type for regular $m$-gonal form, too as follows.

\begin{thm}
\begin{itemize}
\item[(1)] For $m\ge9$ with $m \not\equiv 0 \pmod{4}$ and $m \eq 2 \pmod{3}$, a primitive regular $m$-gonal form $\Delta_{m,\mathbf a}(\mathbf x)$ of rank $n\ge4$ is universal.
\item[(2)] For $m\ge9$ with $m \not\equiv 0 \pmod{4}$ and $m \not\eq 2 \pmod{3}$, a primitive regular $m$-gonal form $\Delta_{m,\mathbf a}(\mathbf x)$ of rank $n\ge4$ is universal or of the form of (\ref{1.notuni/3}).
\item[(3)] For $m\ge12$ with $m \equiv 0 \pmod{4}$ and $m \eq 2 \pmod{3}$, a primitive regular $m$-gonal form $\Delta_{m,\mathbf a}(\mathbf x)$ of rank $n\ge4$ is universal or of the form of (\ref{1.notuni/4m8}), (\ref{3.notuni/4m8}), or (\ref{4.notuni/4m8}).
\item[(4)] For $m\ge12$ with $m \equiv 0 \pmod{4}$ and $m \not\eq 2 \pmod{3}$, a primitive regular $m$-gonal form $\Delta_{m,\mathbf a}(\mathbf x)$ of rank $n\ge4$ is universal or of the form of (\ref{1.notuni/4m8,1m3}), (\ref{3.notuni/4m8,1m3}), (\ref{5.notuni/4m8,1m3}), (\ref{6.notuni/4m8,1m3}), (\ref{7.notuni/4m8,1m3}), (\ref{8.notuni/4m8,1m3}), or (\ref{9.notuni/4m8,1m3}).

\end{itemize}
\end{thm}
\end{rmk}


\begin{thebibliography}{abcd}

\bibitem {O1}  O. T. O'Meara, {\em The Integral representations of quadratic forms over local fields}, Amer. J. of Math, \textbf{80}(1958), 843-878.

\bibitem {O} O. T. O'Meara, {\em Introduction to quadratic forms}, Springer-Verlag, New York, 1963.
\bibitem {D} L. E. Dickson, {\em Ternary quadratic forms and congruences}, Ann. Math. \textbf{28}(1926), 333– 341.
\bibitem {CO} W. K. Chan and B.-K. Oh, {\em Representations of integral quadratic polynomials}, Contemp. Math. \textbf{587} (2013), 31–46.
\bibitem {CR} W. K. Chan and J. Ricci, {\em The representation of integers by positive ternary quadratic polynomials}, J. Number Theory \textbf{156} (2015), 75–94.
\bibitem {HK}  Z. He and B. Kane, {\em Regular Ternary Polygonal Forms}, preprint.
\bibitem {JKS}  W. Jagy, I. Kaplansky, and A. Schiemann, {\em There are 913 regular ternary quadratic forms}, Mathematika \textbf{44} (1997), 332–341. 
\bibitem {JP} B. Jones and G. Pall, {\em  Regular and semi-regular positive ternary quadratic forms}, Acta. Math. \textbf{70} (1939), 165–191. 
\bibitem {M-KO}  M. Kim and B.-K. Oh, {\em Regular ternary quadratic forms}, preprint.
\bibitem {Oliver}  R. Lemke Oliver, {\em Representation by ternary quadratic forms}, Bull. London Math. Soc. \textbf{46} (2014), 1237–1247.
\bibitem {O1}  B.-K. Oh, {\em Regular positive ternary quadratic forms}, Acta Arith. \textbf{147} (2011), 233-243.
\bibitem {O2}  B.-K. Oh, {\em Representations of arithmetic progressions by positive deﬁnite quadratic forms}, Int. J. Number Theory \textbf{7} (2011), 1603–1614. 
\bibitem {O3}  B.-K. Oh, {\em Ternary universal sums of generalized pentagonal numbers}, J. Korean Math. Soc. \textbf{48} (2011), 837-847. 
\bibitem {P}  D. Y. Park, {\em Universal $m$-gonal forms}, Seoul National University, Ph.D. thesis, Seoul National University, 2020. 
\bibitem {W1} G. L. Watson, {\em Some problems in the theory of numbers}, Ph.D. Thesis, University of London, 1953.  
\bibitem {W2} G. L. Watson, {\em Regular positive ternary quadratic forms}, J. Lond. Math. Soc. \textbf{13} (1976), 97-102

\end{thebibliography}
\end{document}